\documentclass[english,british]{article}
\usepackage[T1]{fontenc}
\usepackage[latin9]{inputenc}
\usepackage{geometry}
\usepackage{color}
\usepackage{amsmath}
\usepackage{amssymb}

\makeatletter

\usepackage{amsthm}
\numberwithin{equation}{section}

\newtheorem{Theorem}{Theorem}[section]
\newtheorem*{Theorem*}{Theorem}
\newtheorem{Corollary}[Theorem]{Corollary}

\newtheorem{Proposition}[Theorem]{Proposition}
 { \theoremstyle{definition}
\newtheorem{Definition}[Theorem]{Definition}

\newtheorem{Example}[Theorem]{Example}
\newtheorem{Remark}[Theorem]{Remark} }

\newcommand{\R}{\mathbb{R}}

\newcommand{\be}{\begin{equation}}
\newcommand{\ee}{\end{equation}}
\newcommand{\beq}{\begin{eqnarray}}
\newcommand{\eeq}{\end{eqnarray}}
\newcommand{\w}{\omega}
\newcommand{\s}{\epsilon}

\usepackage{babel}

\makeatother

\usepackage{babel}
\begin{document}

\title{On Third-Order Evolution Systems Describing Pseudo-Spherical or Spherical Surfaces}

\author{Filipe Kelmer$^{1}$}
\date{}
\maketitle
\footnotetext[1]{Departamento de Geometría y Topología, Universidad de Almer\'\i a, Spain,  e-mail: kelmer.a.f@gmail.com. Partially supported by IMAG Grants Program, modality
\textit{Support for visits of young talented researchers} (2023 call) - Institute of Mathematics of the University of Granada (IMAG).}

\begin{abstract}
We consider a class of third-order evolution equations of the form
\begin{equation*}
\left\{
 \begin{array}{l}
 \displaystyle u_{t}=F\left(x,t,u,u_x,u_{xx},u_{xxx},v,v_x,v_{xx},v_{xxx}\right),\\
 \displaystyle v_{t}=G\left(x,t,u,u_x,u_{xx},u_{xxx},v,v_x,v_{xx},v_{xxx}\right),
 \end{array}
 \right. 
\end{equation*}
describing pseudos-pherical (\textbf{pss}) or spherical surfaces (\textbf{ss}), meaning that, their generic solutions $(u(x,t), v(x,t))$ provide metrics, with coordinates $(x,t)$, on open subsets of the plane, with constant curvature $K=-1$ or $K=1$. These systems can be described as the integrability conditions of $\mathfrak{g}$-valued linear problems, with $\mathfrak{g}=\mathfrak{sl}(2,\R)$ or $\mathfrak{g}=\mathfrak{su}(2)$, when $K=-1$, $K=1$, respectively. We obtain characterization and also classification results. Applications of these results provide new examples and new families of such systems, which also contain systems of coupled KdV and mKdV-type equations and nonlinear Schr\"odinger equations. Additionally, this theory is applied to derive a B\"acklund transformation for the coupled KdV system.
\end{abstract}
\noindent 2010 \foreignlanguage{english}{\textit{Mathematics Subject
Classification}: 35G50, 53B20, 58J60, 35Q53,35Q55}

\selectlanguage{english}%


\section{Introduction}

We consider systems of evolution equations describing pseudo-spherical surfaces (\textbf{pss}) or spherical surfaces (\textbf{ss}) of type
\begin{equation}\label{eq:S}
\left\{
 \begin{array}{l}
 \displaystyle u_{t}=F\left(x,t,u,u_x,u_{xx},u_{xxx},v,v_x,v_{xx},v_{xxx}\right),\\
 \displaystyle v_{t}=G\left(x,t,u,u_x,u_{xx},u_{xxx},v,v_x,v_{xx},v_{xxx}\right),
 \end{array}
 \right. 
\end{equation}
for $F$ and $G$ smooth functions. This includes systems of coupled KdV and mKdV-type equations \eqref{eq:coupledKdVEx}, \eqref{eq:mKdV-}, and \eqref{eq:mKdV+}, as well as third-order nonlinear Schr\"odinger equations \eqref{eq:3NLS+} and \eqref{eq:3NLS-}.

Systems of partial differential equations describing \textbf{pss} or \textbf{ss} are characterized by the fact that their generic solutions provide metrics on non-empty open subsets of $\mathbb{R}^{2}$, with Gaussian curvature $K=-1$ or $K=1$, respectively. This concept was first introduced, in 2002, by Q. Ding e K. Tenenblat in \cite{ding} as a generalization of the notion of differential equations describing \textbf{pss} given in 1986 by S. S. Chern and K. Tenenblat in \cite{chern}.

The definition given by Chern and Tenenblat was inspired by Sasaki's observation in 1979 \cite{sasaki}, that a class of nonlinear differential equations, such as KdV, mKdV and SG which can be solved by the AKNS $2\times 2$ inverse scattering method \cite{akns}, was related to \textbf{pss}. Today, it is known that the class of differential equations describing \textbf{pss} is, in fact, larger than the AKNS class. Examples include \cite{chern} Ex. 2.7 a),b), \cite{beals1989} equation (1.2) when $\alpha\neq 0$, \cite{Jorge1987} equations (3.11),(3.13).
 
Besides introducing the notion of differential equations describing pseudo-spherical surfaces (pss), Chern and Tenenblat developed a systematic procedure for characterizing and classifying such equations. This procedure has been applied to various classes of partial differential equations, see 
\cite{castrosilva2015, Catalano2020, Catalano2016, Catalano2014, chern, gomesneto2010,  Jorge1987, kamran1995,reyes1998, Rabelo1989, Rabelo1990, Rabelo1992, reyes1998, Catalano2024}.

Equations describing \textbf{pss} (resp. \textbf{ss}) can be seen as a compatibility condition of an associated $\mathfrak{sl}(2, \mathbb{R})$-valued (resp. $\mathfrak{su}(2)$-valued) linear problem, also referred to as a zero curvature representation. As a consequence, these equations may admit B{\"a}cklund transformations, as demonstrated in Section \ref{sec:BT} for a coupled KdV-type system (see also \cite{chern,beals1989,Reyes-backlund} for other examples). This characterization also suggests that these equations might exhibit additional properties, such as non-local symmetries \cite{Rey7, Rey8}, an infinite number of conservation laws \cite{Cavalcante1988}, and they are natural candidates for being solved by the Inverse Scattering Method \cite{akns,beals1989}.

Another remarkable property of partial differential equations describing \textbf{pss} (resp. \textbf{ss}) is the theoretical existence of local transformations between generic solutions of those equations. This is due to a basic geometric fact that given two points of two Riemannian manifolds, with the same dimension and same constant sectional curvature, there is always an isometry between neighborhoods of those points. Kamran and Tenenblat, in \cite{kamran1995}, explored this property and demonstrated a local existence theorem, assuring that, given any two equations describing \textbf{pss}, then, under a technical assumption, there exists a local smooth application that maps generic solutions of one equation into solutions of the other. Some results in this direction were considered later in \cite{Reyes-backlund}. 

The properties of equations that describe \textbf{pss} or \textbf{ss} are intrinsic geometric properties of a Riemannian metric. As a result, the B{\"a}cklund transformations derived from this geometric framework, such as the transformation for a coupled KdV-type system presented in Theorem \ref{th:BT}, are not initially associated with surfaces of negative Gaussian curvature in $\mathbb{R}^3$ (see also \cite{CamposTenenblat,Campos,BealsTenenblat1991} for the case of pseudo-Riemannian manifolds with constant sectional curvature). This contrasts with the extrinsic geometrical approach by A. V. B{\"a}cklund \cite{Backlund1985,Backlund1905}, which was later extended to surfaces with non-zero constant Gaussian curvature in pseudo-Euclidean space \cite{McNertney1980,Palmer1990,Gu-Hu-Inoguchi2002,KelmerRodrigues2022} and to higher-dimensional submanifolds of constant sectional curvature \cite{TenenblatTerng,Terng,Tenenblat1985,BarbosaFerreiraTenenblat,ChenZuoCheng2004,DajczerTojeiro}. Regardless of the intrinsic or extrinsic approach, the analytical interpretation of B{\"a}cklund transformations, in terms of partial differential equations, enables the generation of new solutions by solving a first-order system (see, for instance, the coupled KdV system \eqref{eq:KdVBT} and the first-order system \eqref{eq:BT}). Furthermore, if a superposition formula is established, it allows for the algebraic construction of infinitely many new solutions \cite{bianchi1,KelmerTenenblat2024}.

In 2002, Ding and Tenenblat \cite{ding}, besides introducing the notion of systems of partial equations describing \textbf{pss} and \textbf{ss}, they presented characterization results for systems of evolution equations of type
 \beq\nonumber
 \left\{
 \begin{array}{l}
 u_t=F(u,u_x,v,v_x),\\
 v_t=G(u,u_x,v,v_x),
 \end{array}
 \right. 
 \eeq
for $F$ e $G$ smooth functions. In particular, they considered important equations such as nonlinear Schr{\"o}dinger equation, Heisenberg ferromagnetic model and Landau-Lifschitz equations. They also presented a classification result for systems of evolution equations of type
 \beq\nonumber
\left\{
 \begin{array}{l}
 u_t=-v_{xx}+H_{11}(u,v)u_x+H_{12}(u,v)v_x+H_{13}(u,v),\\
 v_t=u_{xx}+H_{21}(u,v)u_x+H_{22}(u,v)v_x+H_{23}(u,v),
 \end{array}
 \right. 
 \eeq
where, $H_{ij}$ are smooth functions.

In 2022, Kelmer and Tenenblat \cite{kelmer2022} obtain characterization and classification results for systems of hyperbolic equations of type
 \begin{equation*}
\left\{
 \begin{array}{l}
 \displaystyle u_{xt}=F\left(u,u_x,v,v_x\right),\\
 \displaystyle v_{xt}=G\left(u,u_x,v,v_x\right),
 \end{array}
 \right. 
\end{equation*}
where $F$ and $G$ are smooth functions. By imposing certain conditions, they obtain some classification results providing new examples and new families of systems of differential equations, which contain generalizations of a Pohlmeyer-Lund-Regge type system and the Konno-Oono coupled dispersionless system.

This paper is organized as follows. In Section \ref{sec:Preliminaries},
we gather some preliminaries on systems of differential equations that describe \textbf{pss} or \textbf{ss}, and provide the linear problem associated with these systems, which can be used to obtain B\"{a}cklund transformation. In Section \ref{section-mainresults}, wwe begin by presenting explicit examples, including systems of coupled KdV and mKdV-type equations and nonlinear Schr\"odinger equations. We then state our main results: a characterization result (Theorem \ref{Lemma}), classification results (Theorems \ref{th:f21=l}, \ref{th:f31=l}, \ref{thm:L-f21} and \ref{thm:L-f31}) and, and, as an application of the classification results, we present corollaries that provide additional examples and introduce new families of systems of differential equations describing \textbf{pss} and \textbf{ss}. Section \ref{sec:proofs} is dedicated to proving Theorems \ref{Lemma}, \ref{th:f21=l}, \ref{th:f31=l}, \ref{thm:L-f21} and \ref{thm:L-f31}. Finally, in Section \ref{sec:BT}, we use the associated linear problem to provide a B{\"a}cklund transformation for the coupled KdV system \eqref{eq:coupledKdV} illustrated with examples.

\section{\label{sec:Preliminaries}Preliminaries}
If $\left(M,\, \mathbf{g}\right)$ is a 2-dimensional Riemannian manifold and $\left\{ \omega_{1},\omega_{2}\right\} $ is a co-frame, dual to an orthonormal frame $\left\{ e_{1},e_{2}\right\} $, then the metric is given by  $\mathbf{g}=\omega_{1}^{2}+\omega_{2}^{2}$ and $\omega_{i}$ satisfy the structure equations: $d\omega_{1}=\omega_{3}\wedge\omega_{2}$ and $d\omega_{2}=\omega_{1}\wedge\omega_{3}$, where $\omega_{3}$ denotes the connection form defined as $\omega_{3}(e_{i})=d\omega_{i}(e_{1},e_{2})$. The Gaussian curvature of $M$ is the function $K$ such that $d\omega_{3}=-K\omega_{1}\wedge\omega_{2}$.

\begin{Definition}
A system of partial differential equations $\mathcal{S}$, for scalar functions $u\left(x,t\right)$ and $v\left(x,t\right)$ \emph{describes
pseudospherical surfaces} \textbf{(pss)}\emph{, or spherical surfaces
}\textbf{(ss)} if it is equivalent to the structure equations (see \cite{keti})
of a surface with Gaussian curvature $K=-\delta$, with $\delta=1$
or $\delta=-1$, respectively, i.e., 
\begin{equation}\label{eq:SE}
\begin{array}{l}
d\omega_{1}=\omega_{3}\wedge\omega_{2},\quad d\omega_{2}=\omega_{1}\wedge\omega_{3},\quad d\omega_{3}=\delta\omega_{1}\wedge\omega_{2},\end{array}
\end{equation}
where $\left\{ \omega_{1},\omega_{2},\omega_{3}\right\} $ are $1$-forms
\begin{equation}
\begin{array}{l}
\omega_{1}=f_{11}dx+f_{12}dt,\quad\omega_{2}=f_{21}dx+f_{22}dt,\quad\omega_{3}=f_{31}dx+f_{32}dt,\end{array}\label{eq:forms}
\end{equation}

\noindent such that $\omega_{1}\wedge\omega_{2}\neq0$, i.e., 
\begin{equation}
f_{11}f_{22}-f_{12}f_{21}\neq0,\label{eq:nondeg_cond}
\end{equation}
and $f_{ij}$ are functions of $x$, $t$, $u(x,t)$, $v(x,t)$ and it's derivatives with respect to $x$ and $t$.
\end{Definition}

Locally, considering the basis $\left\{ dx,dt\right\} $, the structure equations (\ref{eq:SE}) are equivalent to the following system of equations
\beq\label{eq:SELocal}
\left\{
\begin{array}{l}
-f_{11,t}+f_{12,x}=f_{31}f_{22}-f_{32}f_{21},\\
-f_{21,t}+f_{22,x}=f_{11}f_{32}-f_{12}f_{31},\\
-f_{31,t}+f_{32,x}=\delta(f_{11}f_{22}-f_{12}f_{21}),
\end{array}
\right.
\eeq
and the metric is
\be\nonumber
ds^2=(f_{11}^2+f_{21}^2)dx^2+2(f_{11}f_{12}+f_{21}f_{22})dxdt+(f_{12}^2+f_{22}^2)dt^2.
\ee
Notice that, according to the definition, given a solution $u,v$ of a system $\mathcal{S}$ describing \textbf{pss} (or \textbf{ss}) with associated 1-forms $\omega_{1},\omega_{2}$ and $\omega_{3}$, we consider an open connected set $U\subset\mathbb{R}^{2}$, contained in the domain of $u,v$, where $\omega_{1}\wedge\omega_{2}$ is everywhere nonzero on $U$. Such an open set $U$ exists for generic solutions $u,v$. Then, $\mathbf{g}=\omega_{1}^{2}+\omega_{2}^{2}$ defines a Riemannian metric, on $U$, with Gaussian curvature $K=-\delta$. It is in this sense that one can say that a system describes, \textbf{pss} (resp. \textbf{ss}). This is an intrinsic geometric property of a  Riemannian metric (not immersed in an ambient space).
For example, the nonlinear Schr{\"o}dinger equation,
\begin{equation}\label{eq:NLSE-}
\left\{
\begin{array}{l}
u_t+v_{xx}-2(u^2+v^2)v=0,\\
-v_t+u_{xx}-2(u^2+v^2)u=0,
\end{array}
\right. 
\end{equation}
describes \textbf{pss} with associated functions
\begin{equation}\label{eq:NLSE-fij}
\begin{array}{lll}
f_{11}=2u,& f_{21}=-2v,& f_{31}=2\eta\\
f_{12}=-4\eta u-2v_{x}, & f_{22}=4\eta v-2u_{x}, & f_{32}=-4\eta^2-2(u^2+v^2),
\end{array}
\end{equation}
where $\eta$ is a parameter. Indeed, for the functions $f_{ij}$ above, the structure equations (\ref{eq:SELocal}), with $\delta=1$, is satisfied modulo (\ref{eq:NLSE-}). In this example, for every generic solution $u(x,t),v(x,t)$, of (\ref{eq:NLSE-}), such that condition (\ref{eq:nondeg_cond})
holds, i.e.,  $-2(u^2+v^2)_x\neq 0$, there exists a Riemannian metric $\mathbf {g}$, with constant Gaussian curvature $K=-1$, whose coefficients $g_{ij}$ are  given by 
\begin{equation*}
\begin{array}{l}
g_{11}=4(u^2+v^2),\\
g_{12}=g_{21}=-16\eta(u^2+v^2)-8(uv_x-vu_x),\\
g_{22}=16\eta^2(u^2+v^2)+4(u_x^2+v_x^2)+16\eta(uv_x-vu_x).
\end{array}
\end{equation*}
In the former example, the parameter $\eta$ implies the existence of a one-parameter family of metrics associated with every generic solution. 


The coupled KdV type system \cite{Hirota1981}, is a third-order evolution system given by
 \begin{equation}\label{eq:coupledKdV}
     \left\{\begin{array}{rl}
       u_t=   &-u_{xxx}+6 uvu_x, \\
       v_t=   & -v_{xxx}+6uvv_x,
     \end{array}
     \right.
 \end{equation}
that describes \textbf{pss} with associated functions
 \begin{equation}\label{eq:coupledKdVfij}
 \begin{array}{ll}
 f_{11}=u+v,&\qquad f_{12}=-u_{xx}-v_{xx}+2(vu^2+uv^2)-\eta^2(u+v)+\eta(v_x-u_x),\\
 f_{21}=\eta,&\qquad f_{22}=-\eta^3+2(vu_x-uv_x)+2\eta uv,\\
 f_{31}=v-u ,& \qquad f_{32}=u_{xx}-v_{xx}+2(uv^2-vu^2)+\eta^2(u-v)+\eta(u_x+v_x),
 \end{array}
 \end{equation}
 where $\eta$ is a parameter.

System of differential equation describing \textbf{pss}, or \textbf{ss}, can be seen as the integrability condition
\be \label{eq:CondIntegrabilidade}
 d\Omega-\Omega\wedge \Omega=0,
\ee
of the linear system \cite{chern},
 \be\label{eq:ProblemaLinear}
 d\Psi=\Omega \Psi,
 \ee
where $\Psi=\left(\Psi_1\\ \Psi_2\right)^T$, with $\Psi_i(x,t)$, $i=1,2,$ and $\Omega$ is either the $\mathfrak{sl}\left(2,\mathbb{R}\right)$-valued
1-form
\[
\Omega=\frac{1}{2}\left(\begin{array}{cc}
\omega_{2} & \omega_{1}-\omega_{3}\\
\omega_{1}+\omega_{3} & -\omega_{2}
\end{array}\right),\qquad\text{when \;}\delta=1,
\]
or the $\mathfrak{su}\left(2\right)$-valued 1-form 
\[
\Omega=\frac{1}{2}\left(\begin{array}{cc}
i\omega_{2} & \omega_{1}+i\omega_{3}\\
-\omega_{1}+i\omega_{3} & -i\omega_{2}
\end{array}\right),\qquad\text{when \;}\delta=-1.
\]
This characterization is due to the fact the structure equation (\ref{eq:SE}) is equivalent to (\ref{eq:CondIntegrabilidade}).

Locally, with $\Omega=A dx+B dt$, the linear problem (\ref{eq:ProblemaLinear}), is given by
\be\label{eq:ProblemaLinearlocal}
\Psi_x=A\Psi,\qquad\Psi_t=B \Psi,
\ee
where
\be\label{eq:ABlocal}
A=\frac{1}{2}\left(\begin{array}{cc}
f_{21}&f_{11}-f_{31}\\
f_{11}+f_{31}&-f_{21}
\end{array}\right),
\quad
B=\frac{1}{2}
\left(\begin{array}{cc}
f_{22}&f_{12}-f_{32}\\
f_{12}+f_{32}&-f_{22}
\end{array}\right),
\ee
when $\delta =1$ or,
\be\label{eq:ABsu2}
A=\frac{1}{2}\left(\begin{array}{cc}
if_{21}&f_{11}+if_{31}\\
-f_{11}+if_{31}&-if_{21}
\end{array}\right),
\quad
B=\frac{1}{2}
\left(\begin{array}{cc}
if_{22}&f_{12}+if_{32}\\
-f_{12}+if_{32}&-if_{22}
\end{array}\right),
\ee
when $\delta=-1$. Moreover, the integrability condition, $\displaystyle \Psi_{xt}=\Psi_{tx}$, is given by
\be\label{eq:compatibilidade2x2}
 A_t-B_x+AB-BA=0,
\ee
which is also known as $\mathfrak{sl}\left(2,\mathbb{R}\right)$ (or $\mathfrak{su}\left(2\right)$) \emph{zero-curvature representation}.

By introducing a new function, known as  \textit{pseudopotential} \cite{Nucci1987}, defined by
\begin{eqnarray}\nonumber
    \Gamma=\frac{\Psi_2}{\Psi_1},
\end{eqnarray}
the linear problem (\ref{eq:ProblemaLinearlocal}) implies a Riccati type system,
\begin{equation}\label{eq:Gamma1}
    \left\lbrace\begin{array}{c}
         \Gamma_{x}=\frac{1}{2}\left(f_{11}+f_{31}\right)-f_{21}\Gamma +\frac{1}{2}\left(-f_{11}+f_{31}\right)\Gamma^2, \\
         \Gamma_{t}=\frac{1}{2}\left(f_{12}+f_{32}\right)-f_{22}\Gamma +\frac{1}{2}\left(-f_{12}+f_{32}\right)\Gamma^2.
    \end{array}
    \right.
\end{equation}
We find that the integrability condition $\Gamma_{xt}=\Gamma_{tx}$ of the Riccati system \eqref{eq:Gamma1} is equivalent to the structure equations (\ref{eq:SELocal}).

For instance, since the coupled KdV system (\ref{eq:coupledKdV}) describes \textbf{pss} it is equivalent to the integrability condition (\ref{eq:compatibilidade2x2}) of the linear problem (\ref{eq:ProblemaLinearlocal}), with $A$ and $B$ defined by (\ref{eq:ABlocal}) where the associates functions $f_{ij}$ is given by (\ref{eq:coupledKdVfij}). That is, for the linear problem
\begin{eqnarray}\nonumber
\Psi_x&=&\left(\begin{array}{cc}
\frac{\eta}{2} & u\\
v & -\frac{\eta}{2}
\end{array}\right)\Psi,\\\nonumber
\Psi_t&=&
\left(\begin{array}{cc}
-\frac{\eta^3}{2}+(\eta v-v_x)u+vu_x&-\eta^2u-\eta u_x-u_{xx}+2vu^2\\
-\eta^2v+\eta v_x-v_{xx}+2uv^2&\frac{\eta^3}{2}-(\eta v-v_x)u-vu_x
\end{array}\right)
\Psi,
\end{eqnarray}
the integrability condition, $\Psi_{xt}=\Psi_{tx}$, is satisfied if, and only if, $u$ and $v$ is a solution of the coupled KdV type system (\ref{eq:coupledKdV}).

The presence of the parameter $\eta$ in the linear problem suggests that the system is a natural candidate for being solved by the Inverse Scattering Method \cite{akns,beals1989}. Additionally, this system may exhibit other noteworthy properties such as non-local symmetries \cite{Rey7, Rey8}, an infinite number of conservation laws \cite{Cavalcante1988} and B\"{a}cklund transformations \cite{chern,beals1989,Reyes-backlund}. In particular, we provide a detailed example of a B\"{a}cklund transformation for the coupled KdV system in Section \ref{sec:BT} (see Theorem \ref{th:BT}).

\section{Some examples and main results \label{section-mainresults}}
In this section, we first illustrate the class of systems of partial differential equations studied in this paper, with some examples that describes \textbf{pss} or \textbf{ss}. Then, we present the characterization and classification results. To maintain focus on the main results and for the convenience of readers primarily interested in applying the results, we will postpone the proofs to Section \ref{sec:proofs}.
\subsection{\label{sec:Examples}Examples}
In this subsection, we present examples of third-order systems of evolution equations that describe \textbf{pss} or \textbf{ss}. These include a coupled KdV-type system \cite{Hirota1981} and a third-order nonlinear Schr\"odinger equation \cite{KodamaHasegawa1987}.
\begin{Example}\label{3NLS+}
    The nonlinear Schr\"odinger equation type \cite{KodamaHasegawa1987},
\begin{equation}
    i q_t+\alpha(t)(q_{xx}+|q|^2q)+i\beta(t)(q_{xxx}+3|q|^2q_x)=0,
\end{equation}
where $\alpha(t)$ and $\beta(t)$ are real smooth functions of $t$, or equivalently in real form ($q(x,t)=u(x,t)+iv(x,t)$),
\begin{equation}\label{eq:3NLS+}
    \left\{ \begin{array}{l}
         u_t=-\alpha v_{xx}-\alpha(u^2+v^2)v-\beta u_{xxx}-3\beta(u^2+v^2)u_x,  \\
         v_t=\alpha u_{xx}+\alpha(u^2+v^2)u-\beta v_{xxx}-3\beta(u^2+v^2)v_x,
    \end{array}
    \right.
\end{equation}
is a third-order evolution system describing \textbf{pss} with associated functions,
\begin{equation}\begin{array}{l}
f_{11}= u+v,\qquad     f_{21}=v-u,\qquad f_{31}=\eta, \\
f_{12}=-\beta(v_{xx}+u_{xx})-(\beta \eta +\alpha)(v_x-u_x)+(u+v)(\beta(\eta^2-u^2-v^2)+\eta\alpha) ,    \\
f_{22}=-\beta(v_{xx}-u_{xx})+(\beta \eta+\alpha)(u_x+v_x)+ (v-u)(\beta(\eta^2-v^2-u^2)+\eta\alpha),    \\
f_{32}=  -2\beta(vu_x-uv_x)+(\beta\eta+\alpha)(\eta^2-u^2-v^2),
\end{array}
\end{equation}
where $\eta$ is a parameter.
\end{Example}
\begin{Example}\label{3NLS-}
    The nonlinear Schr\"odinger equation type,
\begin{equation}
    i q_t+\alpha(t)(q_{xx}-|q|^2q)+i\beta(t)(q_{xxx}-3|q|^2q_x)=0,
\end{equation}
where $\alpha(t)$ and $\beta(t)$ are real smooth functions of $t$, or equivalently in real form ($q(x,t)=u(x,t)+iv(x,t)$),
\begin{equation}\label{eq:3NLS-}
    \left\{ \begin{array}{l}
         u_t=-\alpha v_{xx}+\alpha(u^2+v^2)v-\beta u_{xxx}+3\beta(u^2+v^2)u_x,  \\
         v_t=\alpha u_{xx}-\alpha(u^2+v^2)u-\beta v_{xxx}+3\beta(u^2+v^2)v_x,
    \end{array}
    \right.
\end{equation}
is a third-order evolution system describing \textbf{ss} with associated functions,
\begin{equation}\begin{array}{l}
f_{11}= u+v,\qquad     f_{21}=v-u,\qquad f_{31}=\eta, \\
f_{12}=-\beta(v_{xx}+u_{xx})-(\beta \eta +\alpha)(v_x-u_x)+(u+v)(\beta(\eta^2+u^2+v^2)+\eta\alpha) ,    \\
f_{22}=-\beta(v_{xx}-u_{xx})+(\beta \eta+\alpha)(u_x+v_x)+ (v-u)(\beta(\eta^2+v^2+u^2)+\eta\alpha),    \\
f_{32}=  2\beta(vu_x-uv_x)+(\beta\eta+\alpha)(\eta^2+u^2+v^2),
\end{array}
\end{equation}
where $\eta$ is a parameter.
\end{Example}

\begin{Example}\label{Ex:familyzy2KdV+}
The coupled mKdV type system,
    \begin{equation}\label{eq:mKdV+}
      \begin{array}{l}
           u_t=-u_{xxx}+ \alpha^2(u^2+v^2)u_{x},  \\
           v_t=-v_{xxx}+ \alpha^2(u^2+v^2)v_{x}, 
      \end{array}  
    \end{equation}
where $\alpha\neq 0$ is a real constant, describes \textbf{pss}  for associated functions,
\begin{align*}\nonumber
&f_{11}=\tfrac{\sqrt{6}}{3}\alpha u,\qquad f_{21}=\tfrac{\sqrt{6}}{3}\alpha v,\qquad f_{31}=\eta,\\
&f_{12}=-\tfrac{\sqrt{6}}{3}\alpha\left(u_{xx}+\eta v_{x}\right)+ \tfrac{\sqrt{6}}{9}\alpha u\left(\alpha^2u^2+\alpha^2v^2+3 \eta^2\right),\\
&f_{22}=-\tfrac{\sqrt{6}}{3}\alpha\left(v_{xx}-\eta u_{x}\right)+ \tfrac{\sqrt{6}}{9}\alpha v\left(\alpha^2u^2+\alpha^2v^2+3 \eta^2\right),\\
&f_{32}=\tfrac{2}{3}\alpha^2(vu_{x}-uv_{x})+\tfrac{\eta}{3} \left(\alpha^2u^2+\alpha^2v^2+3 \eta^2\right).
\end{align*}
\end{Example}
\begin{Example}\label{Ex:familyzy2KdV-}
The coupled mKdV type system,
    \begin{equation}\label{eq:mKdV-}
      \begin{array}{l}
           u_t=-u_{xxx}- \alpha^2(u^2+v^2)u_{x},  \\
           v_t=-v_{xxx}- \alpha^2(u^2+v^2)v_{x}, 
      \end{array}  
    \end{equation}
where $\alpha\neq 0$ is a real constant, describes \textbf{ss} for associated functions,
\begin{align*}\nonumber
&f_{11}=\tfrac{\sqrt{6}}{3}\alpha u,\qquad f_{21}=\tfrac{\sqrt{6}}{3}\alpha v,\qquad f_{31}=\eta,\\
&f_{12}=-\tfrac{\sqrt{6}}{3}\alpha\left(u_{xx}+\eta v_{x}\right)- \tfrac{\sqrt{6}}{9}\alpha u\left(\alpha^2u^2+\alpha^2v^2-3 \eta^2\right),\\
&f_{22}=-\tfrac{\sqrt{6}}{3}\alpha\left(v_{xx}-\eta u_{x}\right)- \tfrac{\sqrt{6}}{9}\alpha v\left(\alpha^2u^2+\alpha^2v^2-3 \eta^2\right),\\
&f_{32}=\tfrac{2}{3}\alpha^2(uv_{x}-vu_{x})-\tfrac{\eta}{3} \left(\alpha^2u^2+\alpha^2v^2-3 \eta^2\right).
\end{align*}
\end{Example}

\begin{Example}\label{ex:coupledKdV}
    The coupled KdV type system \cite{Hirota1981}, \begin{equation}\label{eq:coupledKdVEx}
     \left\{\begin{array}{l}
       u_t=   -u_{xxx}+6 uvu_x, \\
       v_t=    -v_{xxx}+6uvv_x,
     \end{array}
     \right.
 \end{equation}
describes \textbf{pss} with
 \begin{equation}\label{eq:coupledKdVfijEx}
 \begin{array}{ll}
 f_{11}=u+v,&\qquad f_{12}=-u_{xx}-v_{xx}+2(vu^2+uv^2)-\eta^2(u+v)+\eta(v_x-u_x),\\
 f_{21}=\eta,&\qquad f_{22}=-\eta^3+2(vu_x-uv_x)+2\eta uv,\\
 f_{31}=v-u ,& \qquad f_{32}=u_{xx}-v_{xx}+2(uv^2-vu^2)+\eta^2(u-v)+\eta(u_x+v_x),
 \end{array}
 \end{equation}
 where $\eta$ is a parameter.
\end{Example}

\subsection{Main Results}
We introduce the following notation for $u(x,t)$ and $v(x,t)$ and its derivatives,
\beq
u=z_0,\quad \dfrac{\partial^i u}{\partial x^i}=z_i,\quad
v=y_0,\quad \dfrac{\partial^i v}{\partial x^i}=y_i,\quad 1\leq i\leq 3.\nonumber
\eeq
Given a smooth function $Q\left(x,t,z_0,z_1,z_2,y_0,y_1,y_2\right)$, we denote the total derivative of $Q$ with respect to $x$ by $D_xQ$, i.e.
\begin{equation}\label{eq:Dx}
    D_xQ=Q_x+\sum_{i=0}^2Q_{z_i}z_{i+1}+Q_{y_i}y_{i+1}.
\end{equation}
With this notation, the system (\ref{eq:S}) can be expressed as:
\be\label{eq:S'}
\left\{
\begin{array}{l}
z_{0,t}=F\left(x,t,z_0,z_1,z_2,z_3,y_0,y_1,y_2,y_{3}\right),\\
y_{0,t}=G\left(x,t,z_0,z_1,z_2,z_3,y_0,y_1,y_2,y_{3}\right).
\end{array}
\right. 
\ee


We now state our characterization result.

\begin{Theorem}\label{Lemma} The necessary and sufficient conditions for the system (\ref{eq:S'}) to describes \textbf{pss} (resp. \textbf{ss}), with associated functions $f_{ij}=f_{ij}(x,t,z_0,z_1,z_2,z_3,y_0,y_1,y_2,y_{3})$, are
\be\label{eq:lemma1-1}
f_{i1,z_j}=0,\quad f_{i1,y_j}=0,\quad f_{i2,z_3}=0,\quad f_{i2,y_3}=0,\quad 1\leq j\leq 3 ,\quad 1\leq i\leq 3,
\ee
\be\label{eq:lemma1-2}
{\left\vert
\begin{array}{cc}
f_{11,z_0}&f_{11,y_0}\\
f_{21,z_0}&f_{21,y_0}
\end{array}
\right\vert}^2
 +
{\left\vert
\begin{array}{cc}
f_{21,z_0}&f_{21,y_0}\\
f_{31,z_0}&f_{31,y_0}
\end{array}
\right\vert}^2
+
{\left\vert
\begin{array}{cc}
f_{11,z_0}&f_{11,y_0}\\
f_{31,z_0}&f_{31,y_0}
\end{array}
\right\vert}^2\neq 0,
\ee
\be\label{eq:lemma1-3}
-f_{11,z_0}F-f_{11,y_0}G-f_{11,t}+D_xf_{12}-f_{31}f_{22}+f_{32}f_{21}=0,
\ee
\be\label{eq:lemma1-4}
-f_{21,z_0}F-f_{21,y_0}G-f_{21,t}+D_xf_{22}-f_{11}f_{32}+f_{12}f_{31}=0,
\ee
\be\label{eq:lemma1-5}
-f_{31,z_0}F-f_{31,y_0}G-f_{31,t}+D_xf_{32}-\delta f_{11}f_{22}+\delta f_{12}f_{21}=0,
\ee
\be\label{eq:lemma1-6}
f_{11}f_{22}-f_{12}f_{21}\neq 0,
\ee
where $\delta=1$ (resp. $\delta=-1$).
\end{Theorem}
We notice that, when $F$, $G$, and $f_{ij}$ does not depend on $(x,t)$, the above theorem reduces the third-order case of Lemma 1 in \cite{ding}.

The following corollary highlights that the system (\ref{eq:S'}), according to Theorem (\ref{Lemma}), must be linear with respect to the higher-order derivatives $z_3$ and $y_3$.
\begin{Corollary}\label{prop:Slinear}
For a system (\ref{eq:S'}), to describe \textbf{pss} or \textbf{ss} with associated functions $f_{ij}$, satisfying (\ref{eq:lemma1-1})-(\ref{eq:lemma1-6}), it is necessary that
\beq
\left\{
\begin{array}{l}
z_{1,t}=F_1z_3+F_2y_3+F_3,\\
y_{1,t}=G_1z_3+G_2y_3+G_3,
\end{array}
\right.
\eeq
where $F_k,G_k$, $k=1,2,3$, are smooth functions of $(x,t,z_0,z_1,z_2,y_0,y_2,y_2)$.
\end{Corollary} 
\begin{proof}
By taking second order derivatives of equations (\ref{eq:lemma1-3})-(\ref{eq:lemma1-5}) with respect to $z_3$ and $y_3$, follows from equation (\ref{eq:lemma1-2}) that $F$ and $G$ are linear with respect $z_3$ and $y_3$.
\end{proof}


In the following characterizations, we will assume that the system of evolution equations (\ref{eq:S}) can not be reduced to a second-order one, by supposing that $F$ and $G$ satisfy the generic condition,
\be\label{eq:irr}
(F_{z_3}^2+G_{z_3}^2)(F_{y_3}^2+G_{y_3}^2)\neq 0,
\ee
up to a zero-measure subset.

\begin{Theorem}\label{th:f21=l}
The system (\ref{eq:S'}) satisfying (\ref{eq:irr}), describes \textbf{pss} (resp. \textbf{ss}) with associated functions $f_{ij}$, satisfying (\ref{eq:lemma1-1})-(\ref{eq:lemma1-6}), such that $f_{21}=\ell$ or $f_{11}=\ell$, for a smooth function $\ell(x,t)$, if, and only if, it is given by
\begin{equation}\label{eq:S-f21=l}
\left(\begin{array}{c}
z_{0,t}\\
y_{0,t}
\end{array}\right)=\frac{1}{W}
\left(\begin{array}{cc}
h_{y_0} & -g_{y_0}\\
-h_{z_0} & g_{z_0}
\end{array}\right)\left(\begin{array}{l}
D_xP-hq+\ell H-g_t\\
D_xH -\delta q g+\delta\ell P-h_t
\end{array}\right),
\end{equation}
where $\delta=1$ (resp. $-1$), $q(x,t,z_0,z_1,y_0,y_1)$, $g(x,t,z_0,y_0)$, $h(x,t,z_0,y_0)$ and $P(x,t,z_0,z_1,z_2,y_0,y_1,y_2)$ are smooth functions satisfying the following generic conditions: $W:=g_{z_0}h_{y_0}-g_{y_0}h_{z_0}\neq 0$, $\ell P-gq\neq 0$,
\begin{equation}\label{eq:irr-f21=l}
   (P_{z_3}^2+H_{z_3}^2)(P_{y_3}^2+H_{y_3}^2)\neq 0, 
\end{equation}
where,
\begin{equation}\nonumber
    H:=\frac{1}{g}(Ph+D_xq-\ell_t).
\end{equation}
Furthermore, for the case when $f_{21}=\ell$, the associated functions are
\be\label{eq:fij-f21=l}
\begin{array}{lll}
f_{11}=g,& f_{21}=\ell, & f_{31}= h,\\
f_{12}=P,& f_{22}=q,& f_{32}=H.
\end{array}
\ee
For $f_{11}=\ell$, the associated functions are as follows,
\be\label{eq:fij-f11=l}
\begin{array}{lll}
f_{11}=\ell,& f_{21}=g,&f_{31}=- h,\\
f_{12}=q,& f_{22}=P,& f_{32}=-H.
\end{array}
\ee
\end{Theorem}

\begin{Remark}
We observe that, in Theorem (\ref{th:f21=l}) the systems describing \textbf{pss} and \textbf{ss} with associated functions satisfying $f_{21}=\ell$ or $f_{11}=\ell$ are the same. However, it is worthwhile to notice that they may have distinct associated linear problems depending on whether the associated functions $f_{ij}$ are given by (\ref{eq:fij-f21=l}) or (\ref{eq:fij-f11=l}).
\end{Remark}   

\begin{Theorem}\label{th:f31=l}
    The system (\ref{eq:S'}) satisfying (\ref{eq:irr}), describes \textbf{pss} (resp. \textbf{ss}) with associated functions $f_{ij}$, satisfying (\ref{eq:lemma1-1})-(\ref{eq:lemma1-6}), such that $f_{31}=\ell$, for a smooth function $\ell(x,t)$, if, and only if, it is given by
\begin{equation}\label{eq:S-f31=l}
\left(\begin{array}{c}
z_{0,t}\\
y_{0,t}
\end{array}\right)=\frac{1}{W}
\left(\begin{array}{cc}
h_{y_0} & -g_{y_0}\\
-h_{z_0} & g_{z_0}
\end{array}\right)\left(\begin{array}{l}
D_xP+hq-\ell H-g_t\\
D_xH -qg+\ell P-h_t
\end{array}\right),
\end{equation}
where $\delta=1$ (resp. $-1$), $q(x,t,z_0,z_1,y_0,y_1)$, $g(x,t,z_0,y_0)$, $h(x,t,z_0,y_0)$ and $P(x,t,z_0,z_1,z_2,y_0,y_1,y_2)$ are smooth functions satisfying the following generic conditions: $W:=g_{z_0}h_{y_0}-g_{y_0}h_{z_0}\neq 0$, $D_xq\neq \ell_t$,
\begin{equation}\label{eq:irr-f31=l}
   (P_{z_3}^2+H_{z_3}^2)(P_{y_3}^2+H_{y_3}^2)\neq 0, 
\end{equation}
where,
\begin{equation}\nonumber
   H:=\frac{1}{g}(Ph+\delta D_xq-\delta \ell_t).
\end{equation}
Moreover, the associated functions are
\be\label{eq:fij-f31=l}
 \begin{array}{lll}
 f_{11}=g,& f_{21}=h, & f_{31}= \ell,\\
 f_{12}=P,& f_{22}=H,& f_{32}=q.
 \end{array}
\ee
\end{Theorem}
Applications of the Theorems \ref{th:f21=l} and \ref{th:f31=l}  provide new examples and new families of systems of evolution equations. The next corollary is an application of Theorem \ref{th:f31=l}.
\begin{Corollary}\label{cor:NLS}
    The family of nonlinear Schr\"odinger type,
\begin{equation}
    iq_t+\alpha(t)(q_{xx}-\delta k|q|^2q)+i\beta(t)(q_{xxx}-3\delta k|q|^2q_x)=0,
\end{equation}
where, $k>0$ is a constant, $\delta=1$ (resp. $\delta=-1$), $\alpha(t),\beta(t)$ are smooth functions of $t$, or equivalently in real form $(q=z_0+iy_0)$,
\begin{equation}
    \left\{\begin{array}{l}
         z_{0,t}= \beta(t)\left(3\delta k(z_0^2+y_0^2)z_{1}-z_{3}\right)+\alpha(t)\left(\delta k (z_0^2+y_0^2)y_0-y_{2}\right),\\
         z_{0,t}=\beta(t)\left(3\delta k(z_0^2+y_0^2)y_1-y_3\right)+\alpha(t)\left(-\delta k (z_0^2+y_0^2)z_0+z_{2}\right),
    \end{array}\right.
\end{equation}
describes \textbf{pss} (resp. \textbf{ss}) with associated functions,
\begin{equation}\begin{array}{l}

f_{11}= \sqrt{k}(z_0+y_0),\qquad     f_{21}=\sqrt{k}(y_0-z_0),\qquad f_{31}=\eta, \\
f_{12}=-\beta\sqrt{k}(y_2+z_{2})-\sqrt{k}(\beta \eta+\alpha)(y_1-z_{1})+f_{11}(\eta\alpha+\eta^2\beta+\delta \beta k(z_0^2+y_0^2)), \\
f_{22}=-\beta\sqrt{k}(y_2-z_{2})+\sqrt{k}(\beta\eta+\alpha)(y_1+z_{1})+f_{21}(\eta\alpha+\eta^2\beta +\delta\beta k(z_0^2+y_0^2)),\\
f_{32}=2\delta k\beta (y_0z_{1}-z_0y_1)+(\eta\beta+\alpha)(\delta k(z_0^2+y_0^2)+\eta^2),
\end{array}
\end{equation}
where $\eta$ is a parameter.
\end{Corollary}
\begin{proof}
    It follows from Theorem \ref{th:f31=l}, by setting $\ell(x,t)=\eta$, $g=\sqrt{k}(z_0+y_0)$, $h=\sqrt{k}(y_0-z_0)$
\begin{equation}
    \begin{array}{l}
         P=-\beta\sqrt{k}(y_2+z_{2})-\sqrt{k}(\beta \eta+\alpha)(y_1-z_{1})+g(\eta\alpha+\eta^2\beta+\delta k\beta(z_0^2+y_0^2)), \\
         H=-\beta\sqrt{k}(y_2-z_{2})+\sqrt{k}(\beta\eta+\alpha)(y_1+z_{1})+h(\eta\alpha+\eta^2\beta +\delta\beta k(z_0^2+y_0^2)),\\
         Q=2\delta k\beta (y_0z_{1}-z_0y_1)+(\eta\beta+\alpha)(\delta k(z_0^2+y_0^2)+\eta^2).
    \end{array}
\end{equation}
\end{proof}
\begin{Remark}
    The nonlinear Schr\"odinger equations given in Examples \ref{3NLS+} and \ref{3NLS-} are derived from Corollary \ref{cor:NLS} by setting $k=1$, with $\delta=1$ and $\delta=-1$ respectively.
\end{Remark}
Motivated by the coupled KdV system (\ref{eq:coupledKdV}), we aim to provide characterizations for systems of the following type,
\begin{equation}\label{eq:coupledKdV-L}
\left\{
\begin{array}{l}
z_{0,t}=az_3+L_{11}(z_0,y_0)z_1+L_{12}(z_0,y_0)y_1+L_{13}(z_0,y_0),\\
y_{0,t}=by_3+L_{21}(z_0,y_0)z_1+L_{22}(z_0,y_0)y_1+L_{23}(z_0,y_0),
\end{array}
\right.
\end{equation}
where, $a$ and $b$ are constants, such that $a^2+b^2\neq 0$ and $L_{ij}$ are smooth functions of $(z_0,y_0)$. The characterizations will describe whether these systems describes \textbf{pss} or \textbf{ss}, under the assumption that the associated functions $f_{ij}$ do not depend on the independent variables $(x,t)$.

For the remainder of this section, we assume that the associated functions $f_{ij}$ do not depend on the independent variables $x$ and $t$. The first characterization of the system (\ref{eq:coupledKdV-L}) describing \textbf{pss} (resp. \textbf{ss}) considers $f_{21}=\eta$ or $f_{11}=\eta$. It demonstrates that the third-order coefficients of such systems must be equal, $a=b$. Additionally, the classification holds under transformations of the form $(z_0,y_0)\rightarrow(z_0+c_1,y_0+c_2)$, where $c_i$ are constants. This implies that for any choice of $c_1$ and $c_2$, the resulting system for $z_0+c_1$ and $y_0+c_2$ describes \textbf{pss} (resp. \textbf{ss}).

\begin{Theorem}\label{thm:L-f21}
    The system (\ref{eq:coupledKdV-L}) describes \textbf{pss} (resp. \textbf{ss}) with associated functions $f_{ij}$ that are independent of $(x,t)$, where $f_{21}=\eta$ or or $f_{11}=\eta$ if and only if, up transformations of type $(z_0,y_0)\rightarrow(z_0+c_1,y_0+c_2)$, $c_i\in \mathbb{R}$, it is given by
\begin{equation}\label{eq:coupledKdV-f21eta}
\begin{split}
z_{0,t}=&az_3+a\left(y_0q_{y_0}-\delta\eta^2+z_0p_{z_0}+p\right)z_1+z_0a\left(p_{y_0}-q_{y_0}\right)y_1+\tfrac{a\eta}{\gamma}(p-q)q_{y_0},\\
y_{0,t}=&ay_3+y_0a\left(p_{z_0}-q_{z_0}\right)z_1+a\left(z_0q_{z_0}-\delta\eta^2+y_0p_{y_0}+p\right)y_1-\tfrac{a\eta}{\gamma}(p-q)q_{z_0},
\end{split}
\end{equation}
where, $\delta=1$ (resp. $-1$), $p$ is an arbitrary smooth function of $(z_0,y_0)$, $q$ given by
\begin{equation}\nonumber
q=\tfrac{1}{2}(q_2^2-\delta q_1^2)+c,\qquad q_i= a_iz_0+b_iy_0,\qquad i=1,2, 
\end{equation}
for $c,a_i,b_i$, $i=1,2$ constants such that $\gamma =a_1b_2-a_2b_1\neq 0$. Moreover, if $f_{21}=\eta$, the associated functions $f_{ij}$ are given by
\begin{align*}\nonumber
&f_{11}=q_1,\qquad f_{21}=\eta,\qquad f_{31}=q_2,\\
&f_{12}=aa_1z_2+ab_1y_2-a\eta (a_2 z_1+b_2 y_1)+aq_1p,\\
&f_{22}=a\gamma(z_0 y_1-y_0 z_1)+a\eta q,\\
&f_{32}=aa_2z_2+ab_2y_2-a\delta\eta(a_1 z_1+b_1y_1)+aq_2p.
\end{align*}
Otherwise, if $f_{11}=\eta$,
\begin{align*}\nonumber
&f_{11}=\eta,\qquad f_{21}=q_1,\qquad f_{31}=-q_2,\\
&f_{12}=a\gamma(z_0 y_1-y_0 z_1)+a\eta q,\\
&f_{22}=aa_1z_2+ab_1y_2-a\eta (a_2 z_1+b_2 y_1)+aq_1p,\\
&f_{32}=-aa_2z_2-ab_2y_2+a\delta\eta(a_1 z_1+b_1y_1)-aq_2p.
\end{align*}
\end{Theorem}
\begin{Remark}
    The coupled KdV system, in Example \ref{ex:coupledKdV}, is derived from Theorem \ref{thm:L-f21} by selecting the parameters $\delta=1$, $a=a_2=-1$, $a_1=b_1=b_2=1$, $c=\eta^2$ and $p=\eta^2-2z_0y_0$.
\end{Remark}
For the subsequent characterization of system (\ref{eq:coupledKdV-L}), we assume $f_{31}=\eta$. It is noteworthy that, in this scenario, the theorem guarantees $a=b$. Furthermore, the characterization remains valid under translations of the dependent variables $(z_0,y_0)$.
\begin{Theorem}\label{thm:L-f31}
    The system (\ref{eq:coupledKdV-L}) describes \textbf{pss} (resp. \textbf{ss}) with associated functions $f_{ij}$ that are independent of $(x,t)$, where $f_{31}=\eta$ if and only if, up transformations of type $(z_0,y_0)\rightarrow(z_0+c_1,y_0+c_2)$, $c_i\in \mathbb{R}$, it is given by
\begin{equation}\label{eq:coupledKdV-f31eta}
\begin{split}
z_{0,t}=&az_3+a\left(p+z_0p_{z_0}-\delta y_0q_{y_0}+\eta^2\right)z_1+z_0a\left(p_{y_0}+\delta q_{y_0}\right)y_1-\tfrac{a\eta}{\gamma}(p+\delta q)q_{y_0},\\
y_{0,t}=&ay_3+y_0a\left(p_{z_0}+\delta q_{z_0}\right)z_1+a\left(p+y_0 p_{y_0}-\delta z_0q_{z_0}+\eta^2\right)y_1+\tfrac{a\eta}{\gamma}(p+\delta q)q_{z_0},
\end{split}
\end{equation}
where, $\delta=1$ (resp. $-1$), $p$ is an arbitrary smooth function of $(z_0,y_0)$, $q$ given by
\begin{equation}\nonumber
q=\tfrac{1}{2}(q_2^2+q_1^2)+c,\qquad q_i= a_iz_0+b_iy_0,\qquad i=1,2, 
\end{equation}
for $c,a_i,b_i$, $i=1,2$ constants such that $\gamma =a_1b_2-a_2b_1\neq 0$. Moreover, the associates functions $f_{ij}$ are given by
\begin{align*}\nonumber
&f_{11}=q_1,\qquad f_{21}=q_2,\qquad f_{31}=\eta,\\
&f_{12}=aa_1z_2+ab_1y_2+a\eta (a_2 z_1+b_2 y_1)+aq_1p,\\
&f_{22}=aa_2z_2+ab_2y_2-a\eta(a_1 z_1+b_1y_1)+aq_2p,\\
&f_{32}=\delta a\gamma (z_0y_1-y_0z_1) -\delta a \eta q.
\end{align*}
\end{Theorem}
\begin{Corollary}\label{cor:familyzy2KdV}
    The family of third-order evolution equations,
    \begin{equation*}
      \begin{array}{l}
           z_{1,t}=-z_3+\delta \alpha^2(z_0^2+y_0^2)z_1,  \\
           y_{1,t}=-y_3+\delta \alpha^2(z_0^2+y_0^2)y_1, 
      \end{array}  
    \end{equation*}
where, $\alpha\neq 0$ is a real constant and $\delta=1$ (resp. $\delta=-1$) describes \textbf{pss} (resp. \textbf{ss}) for associated functions,
\begin{align*}\nonumber
&f_{11}=\tfrac{\sqrt{6}}{3}\alpha z_0,\qquad f_{21}=\tfrac{\sqrt{6}}{3}\alpha y_0,\qquad f_{31}=\eta,\\
&f_{12}=-\tfrac{\sqrt{6}}{3}\alpha\left(z_2+\eta y_1\right)+\delta \tfrac{\sqrt{6}}{9}\alpha z_0\left(\alpha^2z_0^2+\alpha^2y_0^2+3\delta \eta^2\right),\\
&f_{22}=-\tfrac{\sqrt{6}}{3}\alpha\left(y_2-\eta z_1\right)+\delta \tfrac{\sqrt{6}}{9}\alpha y_0\left(\alpha^2z_0^2+\alpha^2y_0^2+3\delta \eta^2\right),\\
&f_{32}=\tfrac{2}{3}\delta\alpha^2(y_0z_1-z_0y_1)+\delta\tfrac{\eta}{3} \left(\alpha^2z_0^2+\alpha^2y_0^2+3\delta \eta^2\right).
\end{align*}
\end{Corollary}
\begin{Remark}
   Examples \ref{Ex:familyzy2KdV+} and \ref{Ex:familyzy2KdV-} are derived from Corollary \ref{cor:familyzy2KdV} by choosing $\delta=1$ and $\delta=-1$ respectively.
\end{Remark}

\section{Proofs of the Characterization and Classification Results\label{sec:proofs}}

Let $\mathcal{M}\subset\mathbb{R}^{10}$ be an open connected subset with coordinates $(x,t,z_0,z_1,z_2,z_3,y_0,y_1,y_2,y_{3})$. Define the ideal $\mathcal{I}$ in $\mathcal{M}$, generated by the forms $\Lambda^j_i$, $1\leq j\leq 2, 0\leq i\leq 4,$ as follows:
\begin{equation}\label{eq:IdealI}
\begin{array}{ll}
\Lambda_i^1=dz_i\wedge dt-z_{i+1}dx\wedge dt,\quad &\Lambda_i^2=dy_i\wedge dt- y_{i+1}dx\wedge dt,\qquad 0\leq i\leq 2,\\
\Lambda_3^1=dz_3\wedge dx\wedge dt, \quad &\Lambda_3^2=dy_3\wedge dx\wedge dt,\\
\Lambda_4^1=dz_0\wedge dx+F dx\wedge dt,\quad &\Lambda_4^2=dy_0\wedge dx+G dx\wedge dt,\\
\end{array}
\end{equation}
where $F$ and $G$ are real smooth functions in $\mathcal{M}$.
The solutions of the system (\ref{eq:S}) are related to integral manifolds of the ideal $\mathcal{I}$, as stated in the following proposition:
\begin{Proposition}
   Let $\mathcal{I}$ be the ideal defined by (\ref{eq:IdealI}) with $F$ and $G$ real smooth function. Then, $\mathcal{I}$ is integrable. If $(u(x,t),v(x,t))$ is a smooth solution to the system (\ref{eq:S}), then the map
\begin{equation}\label{eq:integralPhi}
    \Phi(x,t)=(x,t,z_0(x,t),...,z_3(x,t),y_0(x,t),..,y_3(x,t)),
\end{equation}
with $z_0=u, z_{i+1}=z_{0,x},y_0=v,y_{i+1}=y_{i,x}$, defines a regular integral manifold of $\mathcal{I}$. Conversely, any regular integral manifold of $\mathcal{I}$ given by
\begin{equation}\label{eq:integralPsi}
    \Psi(p,q)=(x(p,q),t(p,q),z_0(p,q),...,z_3(p,q),y_0(p,q),...,y_3(p,q)),
\end{equation}
with $dx$ and $dt$ linearly independent, determines a local solution of the system (\ref{eq:S}).
\end{Proposition}
\begin{proof}
By taking the exterior differentiation of the ideal $\mathcal{I}$, it follows that,
$$
d\Lambda_i^1=\Lambda_{i+1}^1\wedge dx,\qquad d\Lambda^2_i=\Lambda_{i+1}^2\wedge dx,\qquad i=0,1,
$$
$$
d\Lambda_2^1 =-\Lambda_3^1,\qquad d\Lambda_3^1=0,\qquad
d\Lambda^2_2 =-\Lambda^2_3,\qquad d\Lambda^2_3=0,\\
$$
$$
d\Lambda_4^1=-\sum_{i=0}^2\left(  F_{z_i}\Lambda_i^1\wedge dx+F_{y_{i}} \Lambda_i^2 \wedge dx\right)+F_{z_3}\Lambda_3^1+F_{y_3}\Lambda^2_3,\nonumber
$$
$$
d\Lambda_4^2=-\sum_{i=0}^2\left(  G_{z_i}\Lambda_i^1\wedge dx+G_{y_{i}} \Lambda^2_i \wedge dx\right)+G_{z_3}\Lambda_3^1+G_{y_3}\Lambda^2_3,\nonumber
$$
then, $\mathcal{I}$ is closed under exterior differentiation, that is $d\mathcal{I}\subset\mathcal{I}$, and therefore by Frobenius Theorem it is integrable (see \cite{ivey}).

For the second part of the proposition, let $\Phi$ defined by (\ref{eq:integralPhi}). Notice that $\Phi_x$ and $\Phi$ are linearly independent, so $\Phi$ is regular. Moreover, the pull-back of $\Lambda^j_i$ by $\Phi$,
\begin{equation}\nonumber
\begin{array}{lll}
    \Phi^*\Lambda^1_i=\left(\frac{\partial^i u}{\partial x^i}-z_i\right)dx\wedge dt, & \Phi^*\Lambda^2_i=\left(\frac{\partial^i v}{\partial x^i}-y_i\right)dx\wedge dt, & 0\leq i\leq 2,\\
    \Phi^*\Lambda_3^1=0,& \Phi^*\Lambda^2_3=0,&\\
    \Phi^*\Lambda^1_4=\left(\frac{\partial u}{\partial t}+F\circ \Phi\right)dx\wedge dt,& \Phi^*\Lambda^2_4=\left(\frac{\partial v}{\partial t}+G\circ \Phi\right)dx\wedge dt, &
   \end{array}
\end{equation}
vanishes for every $(u(x,t),v(x,t))$ smooth solution of system (\ref{eq:S}). Therefore, $\Phi$ is a regular integral manifold of $\mathcal{I}$.
Conversely, let $\Psi(p,q)$ defined as in (\ref{eq:integralPsi}), then 
$$
\Psi^*(dx\wedge dt)=(x_pt_q-x_qt_p)dp\wedge dq.
$$
Since $\Phi$ is regular and $dx$, $dt$ are linearly independent, we have $x_pt_q-x_qt_p\neq 0$. Therefore, there exists a local reparametrization $\bar{\Psi}(x,t)$, of  ${\Psi}(p,q)$, that is also an integrable manifold of $\mathcal{I}$. Therefore, the projections,
$$
\bar{\Psi}(x,t)\mapsto (x,t,z_0(x,t)),
$$
$$
\bar{\Psi}(x,t)\mapsto (x,t,y_0(x,t)),
$$
define smooth graphs of a solution to the system (\ref{eq:S}).
\end{proof}

\subsection*{Proof of Theorem \ref{Lemma}}
\begin{proof}
Let $f_{ij}(x,t,z_0,z_1,z_2,z_3,y_0,y_1,y_2,y_3)$ be smooth functions in $\mathcal{M}$ and consider $\omega_i=f_{i1}dx+f_{i2}dt$, $i=1,2,3$, 1-forms locally defined in $\mathcal{M}$. The structure equations for a \textbf{pss} (resp. \textbf{ss}),
\begin{equation}\label{eq:IdealJ}
\begin{array}{l}
d\omega_{1}=W_1dx\wedge dt,\quad d\omega_{2}=W_2dx\wedge dt,\quad d\omega_{3}=W_3dx\wedge dt,\end{array}
\end{equation}
where $W_1=f_{31}f_{22}-f_{32}f_{21}$, $W_2=f_{11}f_{32}-f_{12}f_{31}$ and $W_3=\delta f_{11}f_{22}-\delta f_{12}f_{21}$, and $\delta=1$ (resp. $\delta=-1$) are satisfied if, and only if, for $i=1,2,3$
\be\nonumber
(f_{i2,x}-f_{i1,t}-W_i)dx\wedge dt+\sum_{j=0}^{2}( f_{i1,z_j}dz_j\wedge dx+f_{i1,y_j} dy_j\wedge dx+f_{i2,z_j}dz_j\wedge dt+f_{i2,y_j} dy_j\wedge dt)=0.
\ee
Considering $\Lambda^j_i=0$, for $0\leq j\leq 2,0\leq 
i\leq 3$ in the ideal (\ref{eq:IdealI}) we have,
\begin{eqnarray}
( f_{i2,x}-f_{i1,t}-W_i-Ff_{i1,z_0}-Gf_{i1,y_0}+\sum_{j=0}^2 f_{i2,z_j}z_{z+1}+f_{i2,y_j}y_{j+1} )dx\wedge dt+\nonumber\\
+\sum_{j=1}^3(f_{i1,z_j}dz_j\wedge dx+f_{i1,y_j}dy_j\wedge dx)+f_{i2,z_3}dz_3\wedge dx+f_{i2,y_3}dy_3\wedge dx=0,
\end{eqnarray}
 for $i=1,2,3,$. Equating to zero the coefficients of the independent 2-forms, we obtain (\ref{eq:lemma1-1}), (\ref{eq:lemma1-3}), (\ref{eq:lemma1-4}) and (\ref{eq:lemma1-5}). Moreover, equation (\ref{eq:lemma1-2}) must hold in order to obtain $F$ and $G$ from equations (\ref{eq:lemma1-3})-(\ref{eq:lemma1-5}). Finally, equation (\ref{eq:lemma1-6}) must also hold, because it is locally equivalent to $\w_1\wedge\w_2\neq 0$.

A straightforward computation shows that the conditions (\ref{eq:lemma1-1})-(\ref{eq:lemma1-6}), are also sufficient for the system to describe \textbf{pss} (resp. \textbf{ss}).
\end{proof}

\subsection*{Proof of Theorem \ref{th:f21=l}}
\begin{proof}
It suffices to prove Theorem (\ref{th:f21=l}) for $f_{21}=\ell$. This follows from the fact that the structure equations (\ref{eq:SE}) are invariant under the transformation,
\beq\label{prop:f11}
\begin{array}{ll}
f_{11}\rightarrow f_{21},\qquad &f_{12}\rightarrow f_{22},\nonumber\\
f_{21}\rightarrow f_{11},\qquad &f_{22}\rightarrow f_{12},\nonumber\\
f_{31}\rightarrow -f_{31},\qquad &f_{32}\rightarrow -f_{32}.\nonumber
\end{array}
\eeq
This invariance allows us to derive results for $f_{11}=\ell$ from the case $f_{21}=\ell$. 

We first establish the necessary conditions. By Theorem (\ref{Lemma}), we know that the functions $f_{ij}$ satisfy equations (\ref{eq:lemma1-1})-(\ref{eq:lemma1-6}) on an open connected subset. From equation (\ref{eq:lemma1-1}), $f_{i1}$ depends on $(x,t,z_0,y_0)$ while $f_{i2}$ depends on $(x,t,z_0,z_1,z_2,y_0,y_1,y_2)$ for $1\leq i\leq 3$. The generic condition (\ref{eq:lemma1-2}) reduces to \begin{equation}\label{eq:W-f21=l}
W=f_{11,z_0}f_{31,y_0}-f_{11,y_0}f_{31,z_0}\neq 0,
\end{equation}
therefore, equations (\ref{eq:lemma1-3}) and (\ref{eq:lemma1-5}) are equivalent to,
\begin{equation}\label{eq:profS-f21=l}
    \left(\begin{array}{c}
         z_{0,t}  \\
         y_{0,t} 
    \end{array}\right)=\frac{1}{W}\left(
\begin{array}{cc}
   f_{31,y_0}  & -f_{11,y_0} \\
  -f_{31,z_0}   & f_{11,z_0}
\end{array}\right)\left(
\begin{array}{c}
   D_xf_{12}-f_{11,t}-f_{31}f_{22}+\ell f_{32}    \\
   D_xf_{32}-f_{31,t}-\delta f_{11}f_{22}+\delta \ell  f_{12}
\end{array}\right).
\end{equation}
Moreover, equation (\ref{eq:lemma1-4}) reduces to,
\begin{equation}\label{eq:L2-f21=l}
   D_x f_{22}-\ell_t-f_{11}f_{32}+f_{12}f_{31}=0.
\end{equation}
By taking the derivatives with respect to $z_3$ and $y_3$ of the above equality, we have, $$f_{22,z_2}=f_{22,y_2}=0,$$
hence $f_{22}$ is a function that depends only on $(x,t,z_0,z_1,y_0,y_1)$. Moreover, since $f_{11}\neq 0$ (otherwise $W=0$), it follows from equation (\ref{eq:L2-f21=l}) that,
\begin{equation}\label{eq:profH-f21=l}
    f_{32}=\frac{1}{f_{11}}(D_xf_{22}-\ell_t+f_{12}f_{31}).
\end{equation}
By denoting $q=f_{22}$, $g=f_{11}$, $h=f_{31}$, $P=f_{12}$ and $H=f_{32}$, we have that equation (\ref{eq:W-f21=l}) reads $W=g_{z_0}h_{y_0}-g_{y_0}h_{z_0}\neq 0$. Equation (\ref{eq:profH-f21=l}) is equivalent to $H=1/g(D_xq-\ell_t+hP)$, and the system (\ref{eq:profS-f21=l}) reduces to (\ref{eq:S-f21=l}). Moreover, the generic condition (\ref{eq:lemma1-6}) amounts to $\ell P-gq\neq 0$. Finally, the generic condition (\ref{eq:irr}) reduces to the condition (\ref{eq:irr-f21=l}).

The converse is a straightforward computation.
\end{proof}

\subsection*{Proof of Theorem \ref{th:f31=l}}
\begin{proof}
We first establish the necessary conditions. By Theorem (\ref{Lemma}), we know that the functions $f_{ij}$ satisfy equations (\ref{eq:lemma1-1})-(\ref{eq:lemma1-6}) on an open connected subset. From equation (\ref{eq:lemma1-1}), $f_{i1}$ depends on $(x,t,z_0,y_0)$ while $f_{i2}$ depends on $(x,t,z_0,z_1,z_2,y_0,y_1,y_2)$ for $1\leq i\leq 3$. The generic condition (\ref{eq:lemma1-2}) reduces to \begin{equation}\label{eq:W-f31=l}
W=f_{11,z_0}f_{21,y_0}-f_{11,y_0}f_{21,z_0}\neq 0,
\end{equation}
therefore, equations (\ref{eq:lemma1-3}) and (\ref{eq:lemma1-5}) are equivalent to,
\begin{equation}\label{eq:profS-f31=l}
    \left(\begin{array}{c}
         z_{0,t}  \\
         y_{0,t} 
    \end{array}\right)=\frac{1}{W}\left(
\begin{array}{cc}
   f_{21,y_0}  & -f_{11,y_0} \\
  -f_{21,z_0}   & f_{11,z_0}
\end{array}\right)\left(
\begin{array}{c}
   D_xf_{12}-f_{11,t}-\ell f_{22}+f_{21} f_{32}    \\
   D_xf_{22}-f_{21,t}- f_{11}f_{32}+ \ell  f_{12}
\end{array}\right).
\end{equation}
Moreover, equation (\ref{eq:lemma1-4}) reduces to,
\begin{equation}\label{eq:L3-f31=l}
   D_x f_{32}-\ell_t-\delta f_{11}f_{22}+\delta \ell f_{12}=0.
\end{equation}
By taking the derivatives with respect to $z_3$ and $y_3$ of the above equality, we have, $$f_{32,z_2}=f_{32,y_2}=0,$$
hence $f_{32}$ is a function that depends only on $(x,t,z_0,z_1,y_0,y_1)$. Moreover, since $f_{11}\neq 0$ (otherwise $W=0$), it follows from equation (\ref{eq:L3-f31=l}) that,
\begin{equation}\label{eq:profH-f31=l}
    f_{22}=\frac{1}{f_{11}}(\delta D_xf_{32}-\delta \ell_t+f_{12}f_{21}).
\end{equation}
By denoting $q=f_{32}$, $g=f_{11}$, $h=f_{21}$, $P=f_{12}$ and $H=f_{22}$, we have that equation (\ref{eq:W-f31=l}) reads $W=g_{z_0}h_{y_0}-g_{y_0}h_{z_0}\neq 0$. Equation (\ref{eq:profH-f31=l}) is equivalent to $H=1/g(\delta D_xq-\delta \ell_t+Ph)$, and the system (\ref{eq:profS-f31=l}) reduces to (\ref{eq:S-f31=l}). Moreover, the generic condition (\ref{eq:lemma1-6}) amounts to $D_xq\neq \ell_t$. Finally, the generic condition (\ref{eq:irr}) reduces to the condition (\ref{eq:irr-f31=l}).

The converse is a straightforward computation.

\end{proof}

\subsection*{Proof of Theorem \ref{thm:L-f21}}
\begin{proof}
Similarly to the argument provided in the proof of Theorem (\ref{th:f21=l}), we will just consider the case in which $f_{21}=\eta$.

We first prove the necessary conditions. According to Theorem \ref{Lemma}, the system (\ref{eq:coupledKdV-L}) describes \textbf{pss} (resp. \textbf{ss}), for associated functions $f_{ij}$, independent of $(x,t)$, satisfying $f_{21}=\eta$, provided that $f_{i1}$ depends solely on $(z_0,y_0)$ and $f_{i2}$ is a function of $(z_0,z_1,z_2,y_0,y_1,y_2)$ for $i=1,2,3$, satisfying

\begin{eqnarray}
&f_{11,z_0}f_{31,y_0}-f_{11,z_0}f_{31,z_0}\neq 0,\label{eq:demothmL-W}\\
   & f_{11}f_{22}-f_{12}\eta\neq 0,\label{eq:demothL-Delta12}
\end{eqnarray}
\be\label{eq:L1-f21}
\begin{split}
&-f_{11,z_0}(az_3+L_{11}z_1+L_{12}y_1+L_{13})-f_{11,y_0}(by_3+L_{21}z_1+L_{22}y_1+L_{23})+\\
&+D_xf_{12}-f_{31}f_{22}+\eta f_{32}=0,
\end{split}
\ee
\be\label{eq:L2-f21}
D_xf_{22}-f_{11}f_{32}+f_{12}f_{31}=0,
\ee
\be\label{eq:L3-f21}
\begin{split}
&-f_{31,z_0}(az_3+L_{11}z_1+L_{12}y_1+L_{13})-f_{31,y_0}(by_3+L_{21}z_1+L_{22}y_1+L_{23})+\\
&+D_xf_{32}-\delta f_{11}f_{22}+\delta \eta f_{12}=0,
\end{split}
\ee
where, according to notation (\ref{eq:Dx}), for functions $f_{i2}(z_0,z_1,z_2,y_0,y_1,y_2)$, 
\begin{equation}\nonumber  D_xf_{i2}=f_{i2,z_0}z_1+f_{i2,z_1}z_2+f_{i2,z_2}z_3+f_{i2,y_0}y_1+f_{i2,y_1}y_2+f_{i2,y_2}y_3,\qquad i=1,2,3.
\end{equation}
Differentiating (\ref{eq:L1-f21})-(\ref{eq:L3-f21}) with respect to $z_3$ and $y_3$, it follows that,
\begin{align*}
    &\begin{aligned}
        &f_{12,z_2}=af_{11,z_0},\quad && f_{12,y_2}=bf_{11,y_0},\\
        &f_{22,z_2}=0, && f_{22,y_2}=0,\\
        &f_{32,z_2}=af_{31,z_0}, && f_{32,y_2}=bf_{31,y_0},
    \end{aligned}
\end{align*}
hence, there exist smooth functions $P_1,Q_1,H_1$ depending on $(z_0,z_1,y_0,y_1)$, such that,
\begin{align*}
    &\begin{aligned}
        &f_{12}=af_{11,z_0}z_2+bf_{11,y_0}y_2+P_1,\\
        &f_{22}=Q_1,\\
        &f_{32}=af_{31,z_0}z_2+bf_{31,y_0}y_2+H_1.
    \end{aligned}
\end{align*}
With this notation, equations (\ref{eq:L1-f21})-(\ref{eq:L3-f21}) reduce to linear polynomials in $z_2$ and $y_2$. The coefficients following $z_2$ and $y_2$, which depend only on $(z_0,z_1,y_0,y_1)$, must vanish, namely
\begin{equation}\label{L-f21:P1H1Q1}
\begin{array}{l}
P_{1,z_1}+af_{11,z_0,z_0}z_1+af_{11,z_0,y_0}y_1+a\eta f_{31,z_0}=0,\\
P_{1,y_1}+bf_{11,z_0,y_0}z_1+bf_{11,y_0,y_0}y_1+b\eta f_{31,y_0}=0,\\
Q_{1,z_1}+af_{11,z_0}f_{31}-af_{31,z_0}f_{11}=0,\\
Q_{1.y_1}+bf_{11,y_0}f_{31}-bf_{31,y_0}f_{11}=0,\\
H_{1,z_1}+af_{31,z_0,z_0}z_1+af_{31z_0,y_0}y_1+a\eta\delta  f_{11,z_0}=0,\\
H_{1,y_1}+bf_{31z_0,y_0}z_1+bf_{31y_0,y_0}y_1+b\eta \delta f_{11,y_0}=0.
\end{array}
\end{equation}
Interpreting the above relations as equations for $P_1$, $H_1$, and $Q_1$, we find that $Q_1$ is integrable. The compatibility conditions $P_{1,z_1,y_1}=P_{1,y_1,z_1}$ and $H_{1,z_1,y_1}=H_{1,y_1,z_1}$ yield $(a-b)f_{11,z_0,y_0}=0$ and $(a-b)f_{31,z_0,y_0}=0$, respectively.

Hence, we have two cases: $a = b$ and $a \neq b$. We assert that $a=b$, as stated in the following claim:

\textbf{Claim:} $a=b$.

\textit{Proof:} We prove this claim by contradiction. Suppose, for the sake of contradiction, that $a\neq b$. Hence, the compatibility conditions for $P_1$ and $H_1$ reduce to $f_{11,z_0,y_0}=f_{31,z_0,y_0}=0$. Then, there exist smooth functions $g_i(z_0),h_i(y_0)$, $i=1,2$, such that $f_{11}=g_1(z_0)+h_1(y_0)$ and $f_{31}=g_2(z_0)+h_2(y_0)$. Moreover, $P_1$, $Q_1$ and $H_1$ are given by,
\begin{align*}
    &P_1=-\tfrac{1}{2}g_{1,z_0z_0}az_1^2-\tfrac{1}{2}h_{1,y_0y_0} by_1^2-g_{2,z_0}a\eta z_1-h_{2,y_0}b\eta y_1+P_0,\\
    &Q_1=(-g_{1,z_0}a(g_2+h_2)+g_{2,z_0}a(g_1+h_1))z_1+(-h_{1,y_0}b(g_2+h_2)+h_{2,y_0}b(g_1+h_1))y_1+Q_1,\\
    &H_1=-\tfrac{1}{2}g_{2,z_0z_0}az_1^2-\tfrac{1}{2}h_{2,y_0y_0}by_1^2-\delta a \eta g_{1,z_0}z_1-\delta \eta h_{1,y_0}by_1+H_0,
\end{align*}
where, $P_0,Q_0,H_0$ are smooth functions of $(z_0,y_0)$. With this notation, equation (\ref{eq:L2-f21}) reduces to a second-degree polynomial in $z_1$ and $y_1$ whose coefficients depend on $(z_0,y_0)$. Therefore, for (\ref{eq:L2-f21}) to hold, we require each of these coefficients to vanish. In particular, the coefficient of $z_1y_1$ becomes:
$$
(a-b)(h_{1,y_0}g_{2,z_0}-h_{2,y_0}g_{1,z_0})=0.
$$
On the other hand, from (\ref{eq:demothmL-W}), we have $g_{1,z_0}h_{2,y_0}-h_{1,y_0}g_{2,z_0}\neq 0$. Therefore, $a-b=0$, which is a contradiction. This concludes the proof of the claim.

From the preceding claim establishing $b = a$, the equations in (\ref{L-f21:P1H1Q1}) become integrable for $P_1$, $Q_1$, and $H_1$, yielding
\begin{equation*}
\begin{array}{l}
P_1=-\tfrac{1}{2}f_{11,z_0z_0}az_1^2-f_{11,z_0y_0}az_1y_1-\tfrac{1}{2}f_{11,y_0y_0}ay_1^2-f_{31,z_0}a\eta z_1-f_{31,y_0}a\eta y_1+P_0,  \\
Q_1=a(f_{11}f_{31,z_0}-f_{31}f_{11,z_0})z_1+a(f_{11}f_{31,y_0}-f_{31}f_{11,y_0})y_1+Q_0,\\
H_1=-\tfrac{1}{2}f_{31,z_0z_0}az_1^2-f_{31,z_0y_0}az_1y_1-\tfrac{1}{2}f_{31,y_0y_0}ay_1^2-f_{11,z_0}a\delta\eta z_1-f_{11,y_0}a\delta\eta y_1+H_0,
\end{array}
\end{equation*}
where $P_0$, $Q_0$ and $H_0$ are functions of $(z_0,y_0)$.

With this notation, equations (\ref{eq:L1-f21})-(\ref{eq:L3-f21}) become polynomials in $z_1$ and $y_1$. Equation (\ref{eq:L2-f21}) becomes a second-degree polynomial, while equations (\ref{eq:L1-f21}) and (\ref{eq:L3-f21}) become third-degree polynomials. The coefficients of these polynomials depend only on functions of $(z_0,y_0)$. For these equations to hold, each of these coefficients must vanish. From the third and second-degree terms in (\ref{eq:L1-f21})-(\ref{eq:L3-f21}), we observe that all third and second-order partial derivatives of $f_{11}$ and $f_{31}$ must vanish. Therefore, we express them as first-degree polynomials, denoted by $q_1$ and $q_2$, respectively
$$
f_{11}=q_1=a_1z_0+b_1y_0+a_0,\quad f_{31}=q_2=a_2z_0+b_2y_0+b_0,
$$
User
where ${a}_i$ and ${b}_i$ are constants that, by equation (\ref{eq:demothmL-W}), must satisfy the condition
$$
\gamma:=a_1b_2-a_2b_1\neq 0.
$$
We can assume $a_0=b_0=0$ up to a transformation $(z_0,y_0)\rightarrow (z_0+c_1, y_0+c_2)$, namely
$$
c_1=\frac{b_1b_0-b_2a_0}{\gamma},\qquad c_2=\frac{a_2a_0-a_1b_0}{\gamma},
$$
therefore,
\begin{equation}\label{eq:demoL-f21-q1q2}
q_1=a_1z_0+b_1y_0,\quad q_2=a_2z_0+b_2y_0.    
\end{equation}
We introduce $p$ and $h$ as functions of $(z_0,y_0)$ given by
\begin{equation}\label{eq:L-f21-ph}
p=\frac{-b_1 H_0+b_2 P_0}{z_0a\gamma},\qquad h=\frac{a_1H_0-a_2P_0}{y_0a\gamma},
\end{equation}
and we rewrite the arbitrary functions $P_0$ and $H_0$ as $P_0=aa_1z_0p+ab_1y_0h$ and $H_0=ab_2y_0h+aa_2z_0p$.
With this notation the equation (\ref{eq:L2-f21}), reduces to
\begin{equation}\label{eq:L-f21-L2z1y10}
(Q_{0,z_0}+ ( q_1 a_1\delta- q_2 a_2)a\eta)z_1+(Q_{0,y_0}+ ( q_1 b_1\delta- q_2 b_2)a\eta)y_1+z_0y_0a\gamma (p-h)=0,
\end{equation}
where both the coefficients of $z_1$ and $y_1$, as well as the independent term vanish. For the terms involving $z_1$ and $y_1$ terms, $q_1$ and $q_2$ are given by (\ref{eq:demoL-f21-q1q2}) and satisfy $( q_1a_1\delta-a_2q_2)_{y_0}=( q_1b_1\delta-q_2b_2)_{z_0}$. This allows us to determine $Q_0$ as,
$$
Q_0=a\eta q,\qquad \quad q=\tfrac{1}{2}(q_2^2-\delta q_1^2)+c,
$$
where $c$ is an integration constant. Furthermore, since $\gamma \neq 0$, the independent term implies $h=p$.
With this notation we have,
\begin{align*}
&f_{11}=q_1,\qquad f_{21}=\eta,\qquad f_{31}=q_2,\\
&f_{12}=aa_1z_2+ab_1y_2-aa_2\eta z_1-ab_2\eta y_1+aq_1p,\\
&f_{22}=-ay_0\gamma z_1+az_0\gamma y_1+a\eta q,\\
&f_{32}=aa_2z_2+ab_2y_2-aa_1\delta\eta z_1-a\delta \eta b_1y_1+aq_2p.
\end{align*}
Then, the condition (\ref{eq:demothL-Delta12}) can be expressed as
$$
-a\eta a_1 z_2-a\eta b_1y_2+(-\gamma a y_0q_1+a\eta^2 a_2)z_1+(\gamma z_0 a q_1-a\eta^2 b_2)y_1+a\eta ( q- p)q_1\neq 0.
$$
This condition is satisfied. Otherwise, each coefficient must vanish. Given that $\gamma = a_1b_2 - a_2b_1 \neq 0$, we have that both $a_1^2 + b_1^2 \neq 0$ and $q_1 \neq 0$. From the coefficients of $z_2$ and $y_2$, we deduce $a\eta a_1 = a\eta b_1 = 0$. Since $a \neq 0$, it implies $\eta = 0$. Consequently, the coefficients of $z_1$ and $y_1$ for $\eta=0$ are equivalent to $\gamma a q_1 = 0$, which leads to a contradiction.

Finally, equations (\ref{eq:L1-f21}) and (\ref{eq:L3-f21}) reduce to first-degree polynomials in $z_1$ and $y_1$. Each of these equations requires setting its coefficients to zero, namely
\begin{align}
&-a_1L_{11}-b_1L_{21}-a\delta \eta^2a_1+a\gamma y_0 q_2+aa_1p+aq_1p_{z_0}=0,\label{eq:L-f21-L1-z1} \\
&-a_1L_{12}-b_1L_{22}-a\delta \eta^2b_1-a\gamma z_0 q_2+ab_1p+aq_1p_{y_0}=0,\label{eq:L-f21-L1-y1}  \\
&-a_1L_{13}-b_1L_{23}-a\eta (q-p)q_2=0,\label{eq:L-f21-L1-0} \\
&-a_2L_{11}-b_2L_{21}-aa_2\delta \eta^2+a\gamma \delta y_0q_1+a_2p+aq_2p_{z_0}=0,\label{eq:L-f21-L3-z1}\\
&-a_2L_{12}-b_2L_{22}-ab_2\delta \eta^2-a\gamma \delta z_0q_1+ab_2p+aq_2p_{y_0}=0,\label{eq:L-f21-L3-y1}\\
&-a_2 L_{13}-b_2L_{23}-a\delta\eta (q-p).\label{eq:L-f21-L3-0}
\end{align}
These equations can be solved for $L_{ij}$, $1\leq i\leq 2$, $1\leq j\leq 3$. By using the identities,
\begin{equation*}
\begin{array}{ll}
a_2q_2-\delta a_1q_1=q_{z_0},\qquad& b_2q_2-\delta b_1q_1=q_{y_0},\\
b_2q_1-b_1q_2=\gamma z_0,\qquad &a_2q_1-a_1q_2=-\gamma y_0,
\end{array}
\end{equation*}
we obtain
\begin{align*}
&L_{11}=ay_0q_{y_0}-a\delta\eta^2+az_0p_{z_0}+ap,\qquad L_{12}=-az_0q_{y_0}+ap_{y_0}z_0,\qquad L_{13}=\tfrac{a\eta}{\gamma}(p-q)q_{y_0},\\
&L_{21}=-ay_0q_{z_0}+ay_0p_{z_0},\qquad L_{22}=az_0q_{z_0}-a\delta\eta^2+y_0ap_{y_0}+ap,\qquad L_{23}=-\tfrac{a\eta}{\gamma}(p-q)q_{z_0}.
\end{align*}
Consequently, the system (\ref{eq:coupledKdV-L}) simplifies to (\ref{eq:coupledKdV-f21eta}). This concludes the first part of the proof.

The converse follows from a straightforward computation.
\end{proof}
\subsection*{Proof of Theorem \ref{thm:L-f31}}
\begin{proof}
The proof follows a similar approach to that of Theorem \ref{thm:L-f21}. By solving the equations derived from setting the high-order coefficients of equations (\ref{eq:lemma1-3})-(\ref{eq:lemma1-5}) to zero, we successively obtain a polynomial of lower order. Hence, we will provide a brief outline of the proof.

By Theorem \ref{Lemma}, the necessary conditions the system (\ref{eq:coupledKdV-L}) describes \textbf{pss} (resp. \text{ss}) with associated functions $f_{ij}$, not depending on $(x,t)$, with $f_{31}=\eta$, are
\begin{eqnarray}
&f_{11,z_0}f_{21,y_0}-f_{11,z_0}f_{21,z_0}\neq 0,\label{eq:demothmL-f31W}\\
   & f_{11}f_{22}-f_{12}f_{21}\neq 0,\label{eq:demothL-f31Delta12}
\end{eqnarray}
\be\label{eq:L1-f31}
\begin{split}
&-f_{11,z_0}(az_3+L_{11}z_1+L_{12}y_1+L_{13})-f_{11,y_0}(bz_3+L_{21}z_1+L_{22}y_1+L_{23})+\\
&+D_xf_{12}-\eta f_{22}+f_{21} f_{12}=0,
\end{split}
\ee

\be\label{eq:L2-f31}
\begin{split}
&-f_{21,z_0}(az_3+L_{11}z_1+L_{12}y_1+L_{13})-f_{21,y_0}(bz_3+L_{21}z_1+L_{22}y_1+L_{23})+\\
&+D_xf_{22}-f_{11} f_{32}+\eta f_{32}=0,
\end{split}
\ee

\be\label{eq:L3-f31}
D_xf_{32}-\delta f_{11}f_{22}+\delta f_{21} f_{12}=0,
\ee
where, according to notation (\ref{eq:Dx}), for functions $f_{i2}(z_0,z_1,z_2,y_0,y_1,y_2)$, 
\begin{equation}\nonumber  D_xf_{i2}=f_{i2,z_0}z_1+f_{i2,z_1}z_2+f_{i2,z_2}z_3+f_{i2,y_0}y_1+f_{i2,y_1}y_2+f_{i2,y_2}y_3,\qquad i=1,2,3.
\end{equation}
From $z_3$ and $y_3$ coefficients of (\ref{eq:L1-f31})-(\ref{eq:L3-f31}), we obtain
\begin{align*}
    &\begin{aligned}
        &f_{12}=af_{11,z_0}z_2+bf_{11,y_0}y_2+P_1,\\
        &f_{22}=af_{21,z_0}z_2+bf_{21,y_0}y_2+H_1,\\
        &f_{32}=Q_1,
    \end{aligned}
\end{align*}
where $P_1,Q_1$ and $ H_1$ are smooth functions of $(z_0,z_1,y_0,y_1)$.

Similar to the proof of Theorem (\ref{thm:L-f21}), we can demonstrate that $a=b$. Then, from the coefficients of $z_2$ and $y_2$ in equations (\ref{eq:L1-f31})-(\ref{eq:L3-f31}), we obtain
\begin{equation*}
\begin{array}{l}
P_1=-\tfrac{1}{2}f_{11,z_0z_0}az_1^2-f_{11,z_0y_0}az_1y_1-\tfrac{1}{2}f_{11,y_0y_0}ay_1^2+f_{21,z_0}a\eta z_1+f_{21,y_0}a\eta y_1+P_0,  \\
H_1=-\tfrac{1}{2}f_{21,z_0z_0}az_1^2-f_{21,z_0y_0}az_1y_1-\tfrac{1}{2}f_{21,y_0y_0}ay_1^2-f_{11,z_0}a\eta z_1-f_{11,y_0}a\eta y_1+H_0,\\
Q_1=a\delta (f_{11}f_{21,z_0}-f_{21}f_{11,z_0})z_1+a\delta (f_{11}f_{21,y_0}-f_{21}f_{11,y_0})y_1+Q_0,\\
\end{array}
\end{equation*}
where $P_0$, $Q_0$ and $H_0$ are functions of $(z_0,y_0)$. Moreover, we introduce $p$ and $h$ as functions of $(z_0,y_0)$, such that,
$$P_0=aa_1z_0p+ab_1y_0h,\qquad H_0=ab_2y_0h+aa_2z_0p.$$

From the third and second-degree terms in (\ref{eq:L1-f31})-(\ref{eq:L3-f31}), we have that all second and third-order partial derivatives of $f_{11}$ and $f_{31}$ are zero. Therefore, $f_{11}$ and $f_{31}$ are first-order polynomials
$$
f_{11}=q_1=a_1z_0+b_1y_0,\quad f_{31}=q_2=a_2z_0+b_2y_0,
$$
where ${a}_i$ and ${b}_i$, where $\gamma:=a_1b_2-a_2b_1\neq 0$. Here we assumed, up to a transformation $(z_0,y_0)\rightarrow (z_0+c_1, y_0+c_2)$, that the independent coefficient of $q_1$ and $q_2$ are zero.

With this notation the equation (\ref{eq:L3-f31}), reduces to
\begin{equation}\nonumber
(Q_{0,z_0}+ \delta a\eta ( q_1 a_1+ q_2 a_2))z_1+(Q_{0,y_0}+ \delta a \eta ( q_1 b_1+ q_2 b_2)a\eta)y_1+a\delta\gamma z_0y_0 (p-h)=0,
\end{equation}
them, $h=p$ and
$$
Q_0=-a\eta \delta q,\qquad \quad q=\tfrac{1}{2}(q_2^2+q_1^2)+c,
$$
where $c$ is an integration constant.

Therefore, in this case, the associates functions $f_{ij}$ are given by
\begin{align*}
&f_{11}=q_1,\qquad f_{21}=q_2,\qquad f_{31}=\eta,\\
&f_{12}=aa_1z_2+ab_1y_2+\eta a (a_2 z_1+b_2 y_1)+q_1p,\\
&f_{22}=aa_2z_2+ab_2y_2-\eta a(a_1 z_1+b_1y_1)+q_2p,\\
&f_{32}=\delta\gamma a(z_0y_1-y_0z_1) -\delta \eta a q.
\end{align*}
Moreover, the coefficients of $z_2$ and $y_2$ in (\ref{eq:demothL-f31Delta12}) are $-a\gamma y_0$ and $a\gamma z_0$, respectively, which do not vanish. This is sufficient for condition (\ref{eq:demothL-f31Delta12}) to hold.
Finally, equations (\ref{eq:L1-f31}) and (\ref{eq:L2-f31}) reduce to first-degree polynomials in $z_1$ and $y_1$. Each of these equations requires setting its coefficients to zero, namely
\begin{align}
&-a_1L_{11}-b_1L_{21}+aa_1\eta^2-a\delta a_1q_2^2+a\delta a_2 q_1q_2+aa_1p+aq_1p_{z_0}=0,\nonumber \\
&-a_1L_{12}-b_1L_{22}+a\eta^2 b_1-a\delta b_1q_2^2+a\delta b_2q_1q_2+ab_1p+aq_1p_{y_0}=0,\nonumber \\
&-a_1L_{13}-b_1L_{23}-a\eta (p+\delta q)q_2=0,\nonumber \\
&-a_2L_{11}-b_2L_{21}+aa_2\eta^2-a\delta a_2q_1^2+a\delta a_1q_1q_2+aa_2p+aq_2p_{z_0}=0,\nonumber\\
&-a_2L_{12}-b_2L_{22}+a\eta^2 b_2-a\delta b_2q_1^2+a\delta b_1q_2q_1+ab_2p+aq_2p_{y_0}=0,\nonumber\\
&-a_2 L_{13}-b_2L_{23}+a\eta q_1(p+\delta q).\nonumber
\end{align}
These equations can be solved for $L_{ij}$, $1\leq i\leq 2$, $1\leq j\leq 3$. By using the identities,
\begin{equation*}
\begin{array}{ll}
a_1q_1+ a_2q_2=q_{z_0},\qquad& b_2q_2+ b_1q_1=q_{y_0},\\
b_2q_1-b_1q_2=\gamma z_0,\qquad &a_2q_1-a_1q_2=-\gamma y_0,
\end{array}
\end{equation*}
we obtain,
\begin{align*}
&L_{11}=-a\delta y_0q_{y_0}+a\eta^2+z_0ap_{z_0}+ap,\qquad L_{12}=a\delta z_0q_{y_0}+z_0ap_{y_0},\qquad L_{13}=-\tfrac{\eta}{a\gamma}(p+\delta q)q_{y_0},\\
&L_{21}=a\delta y_0q_{z_0}+y_0ap_{z_0},\qquad L_{22}=-a\delta z_0q_{z_0}+a\eta^2+y_0ap_{y_0}+ap,\qquad L_{23}=\tfrac{a\eta}{\gamma}(p+\delta q)q_{z_0}.
\end{align*}
Consequently, we obtain (\ref{eq:coupledKdV-f31eta}). The converse follows from a straightforward computation.
\end{proof}
\section{B\"{a}cklund Transformation \label{sec:BT} for the coupled KdV system}
In this section we prove a B\"{a}cklund Transformation for the coupled KdV system (Example \ref{ex:coupledKdV}),
\begin{equation}\label{eq:KdVBT}
     \left\{\begin{array}{l}
       u_t=  -u_{xxx}+6 uvu_x, \\
       v_t=  -v_{xxx}+6uvv_x.
     \end{array}
     \right.
\end{equation}
This transformation allows one to construct a four-parameter family of new solutions by solving a first-order system of equations. The B\"{a}cklund Transformation is given in Theorem \ref{th:BT}, with an equivalent version provided in Theorem \ref{th:BTalt}. Additionally, we illustrate this transformation with examples.

\begin{Theorem}\label{th:BT}
Let $(u,v)$ be a real solution of the coupled KdV system. Then, for any constants $\lambda_1,\lambda_2$, such that $\lambda_1\neq 0$, the following system is completely integrable,
\begin{equation}\label{eq:BT}
    \begin{array}{rl}
         (u'-u)_x=&\sigma (u'+u)((u'-u)(v'-v)-\lambda_1^2)^{\frac{1}{2}}+\lambda_2(u'-u), \\
          (v'-v)_x=&\sigma(v'+v)((u'-u)(v'-v)-\lambda_1^2)^{\frac{1}{2}} -\lambda_2(v'-v),\\
          (u'-u)_t=&-2\sigma ((u'-u)(v'-v)-\lambda_1^2)^{\frac{1}{2}}\left[u_{xx}+\lambda_2u_x+\lambda_2^2(u'-u)-(u'+u)\left(uv+\frac{\lambda_1^2-\lambda_2^2}{2}\right)\right]+\\
          &-\left[(u'-u)(v'-3v)-2\lambda_1^2\right]u_x-(u'-u)(u'+u)v_x+\\
          &+2\lambda_2(u'-u)\left(u'v'+\frac{\lambda_1^2-\lambda_2^2}{2}\right)-2\lambda_2(u'+u)\left((u'-u)(v'-v)-\lambda_1^2\right),\\
          (v'-v)_t=&2\sigma ((u'-u)(v'-v)-\lambda_1^2)^{\frac{1}{2}}\left[-v_{xx}+\lambda_2v_x-\lambda_2^2(v'-v)+(v'+v)\left(uv+\frac{\lambda1^2-\lambda_2^2}{2}\right)\right]+\\
          &-\left[(v'-v)(u'-3u)-2\lambda_1^2\right]v_x-(v'-v)(v'+v)u_x+\\
          &-2\lambda_2(v'-v)\left(u'v'+\frac{\lambda_1^2-\lambda_2^2}{2}\right)+2\lambda_2(v'+v)\left((u'-u)(v'-v)-\lambda_1^2\right),
    \end{array}
\end{equation}
where $\sigma=\pm 1$. Moreover, $(u',v')$ is a solution of the coupled KdV system. 
\end{Theorem}
\begin{proof}
We know from Example \ref{ex:coupledKdV} that the coupled KdV system describes \textbf{pss} with associated functions given by \eqref{eq:coupledKdVfijEx}. Therefore, the associated Riccati system \eqref{eq:Gamma1},
\begin{equation*}
    \begin{array}{l}
         \Gamma_x=-u\Gamma^2-\eta\Gamma+v,  \\
         \Gamma_t=(u_{xx}+\eta u_x+\eta^2u-2vu^2)\Gamma^2+\left[2(uv_x-vu_x-\eta uv)+\eta^3\right]\Gamma-v_{xx}+\eta v_x-\eta^2v+2uv^2, 
    \end{array}
\end{equation*}
where $\eta\in\mathbb{C}$, is integrable for $\Gamma(x,t)\in\mathbb{C}$. Considering $\Gamma=\phi +i \psi$ and $\eta=\lambda_2+i\lambda_1$, the real and imaginary parts of the Riccati system read as follows, 
\begin{equation}\label{eq:psiphi}
\begin{split}
\phi_x=&u(\psi^2-\phi^2)+\lambda_1\psi-\lambda_2\phi+v,\\
\psi_x=&-2u\phi\psi-\lambda_1\phi-\lambda_2\psi\\
\phi_t=& \left[u_{xx}+\lambda_2 u_x-2vu^2+(\lambda_2^2-\lambda_1^2)u\right](\phi^2-\psi^2)-2\lambda_1(2\lambda_2 u+u_x)\phi\psi+\\
& +\left[\lambda_2(-2uv-3\lambda_1^2+\lambda_2^2)+2(uv_x-vu_x) \right]\phi+\lambda_1(2uv+\lambda_1^2-3\lambda_2^2)\psi+\\
& -v_{xx}+\lambda_2v_x+2uv^2+(\lambda_1^2-\lambda_2^2)v,\\
\psi_t=&\lambda_1(2\lambda_2u+u_x)(\phi^2-\psi^2)+2\left[u_{xx}+\lambda_2 u_x-2vu^2+(\lambda_2^2-\lambda_1^2)u\right]\phi\psi+\\
 &-\lambda_1(2uv+\lambda_1^2-3\lambda_2^2)\phi+\left[-\lambda_2(2uv+3\lambda_1^2-\lambda_2^2)+2(uv_x-vu_x)\right]\psi+\\
 &+ \lambda_1(v_x-2v\lambda_2).
\end{split}
\end{equation}
From the first two equations of \eqref{eq:psiphi}, we obtain,
\begin{equation}\label{eq:uvphipsix}
\begin{array}{l}
u=-\dfrac{\psi_x+\lambda_1\phi+\lambda_2\psi}{2\psi\phi},\\
v=\phi_x-\dfrac{(\phi^2-\psi^2)\psi_x+(\phi^2+\psi^2)(\lambda_1\phi-\lambda_2\psi)}{2\psi\phi}.
\end{array}
\end{equation}
By replacing $u,v$ into the last two equations of \eqref{eq:psiphi}, we have
\begin{equation}\label{eq:Hphipsit}
\begin{split}
    \phi_t=-\phi_{xxx}+\frac{3}{2}\phi_xA-\frac{3}{2}\dfrac{(\psi_x^2+\phi_x^2)(2\lambda_2\phi-\phi_x)}{\phi^2}+3\dfrac{\phi\phi_x-\psi\psi_x}{\psi\phi}\psi_{xx},\\
     \psi_t=-\psi_{xxx}+\frac{3\psi_x}{2}A-\frac{3}{2}\dfrac{(\psi_x^2+\phi_x^2)(2\lambda_2\psi+\psi_x)}{\phi^2}+3\dfrac{\phi\psi_x+\psi\phi_x}{\psi\phi}\psi_{xx},
\end{split}
\end{equation}
where,
\begin{equation*}
  A=\dfrac{(\psi^2+\phi^2)(\lambda_1^2\phi^2-\lambda_2^2\psi^2)-(\phi\psi_x+\psi\phi_x)^2}{\phi^2\psi^2}  .
\end{equation*}
Hence, the system \eqref{eq:psiphi} is equivalent to \eqref{eq:uvphipsix}-\eqref{eq:Hphipsit}. It follows that, given a solution $(u,v)$ of the coupled KdV system, then for each pair of constants $\lambda_1$ and $\lambda_2$, the system \eqref{eq:psiphi} is completely integrable, and $(\phi,\psi)$ is a solution of \eqref{eq:Hphipsit}. Conversely, if $(\phi,\psi)$ is a solution of \eqref{eq:psiphi}, then $(u,v)$ defined by \eqref{eq:uvphipsix} satisfies the coupled KdV system.

We notice that equation \eqref{eq:Hphipsit} is invariant under the transformation $$(\phi,\psi,\lambda_1,\lambda_2)\mapsto (\phi',\psi',\lambda_1',\lambda_2')=(\phi,\psi,-\lambda_1,\lambda_2).$$ As a consequence, if $(u,v)$ is a solution of the coupled KdV and $(\phi,\psi,\lambda_1,\lambda_2)$ satisfies \eqref{eq:psiphi}, then the functions $(u',v')$, given by
\begin{equation}\label{eq:uvb}
\begin{array}{l}
u'=-\dfrac{\psi_x-\lambda_1\phi+\lambda_2\psi}{2\psi\phi},\\
v'=\phi_x-\dfrac{(\phi^2-\psi^2)\psi_x-(\phi^2+\psi^2)(\lambda_1\phi+\lambda_2\psi)}{2\psi\phi},
\end{array}
\end{equation}
is also a solution of the coupled KdV system. From \eqref{eq:uvb} and \eqref{eq:uvphipsix} we obtain
\begin{equation}\label{eq:ubuvbv}
    u'-u=\frac{\lambda_1}{\psi},\qquad v'-v=\lambda_1\frac{\phi^2+\psi^2}{\psi}.
\end{equation}
From \eqref{eq:uvb}, we have that
\begin{equation}\label{eq:psiphiuv}
         \psi=\dfrac{\lambda_1}{u'-u},  \qquad
         \phi=\sigma\dfrac{((u'-u)(v'-v)-\lambda_1^2)^{\frac{1}{2}}}{u'-u},
\end{equation}
where $\sigma=\pm 1$.

Finally, differentiating \eqref{eq:ubuvbv} with respect to $x$ and $t$ and using \eqref{eq:psiphi} and \eqref{eq:psiphiuv}, we have \eqref{eq:BT} which concludes the proof.
\end{proof}
\begin{Example}
Let $u=v\equiv 0$ the trivial solution of the coupled KdV system, then
\begin{equation}\label{eq:BTuv0}
    \begin{array}{l}
         u'=\sigma\lambda_1 e^{\lambda_2x-(\lambda_2^3-3\lambda_2\lambda_1^2)t-k_2}\sec(\lambda_1x+(\lambda_1^3-3\lambda_1\lambda_2^2)t+k_1),  \\
         v'=\sigma\lambda_1e^{-\lambda_2x+(\lambda_2^3-3\lambda_2\lambda_1^2)t+k_2}\sec(\lambda_1x+(\lambda_1^3-3\lambda_1\lambda_2^2)t+k_1).
    \end{array}
\end{equation}
where $\sigma=\pm 1$ and $\lambda_i,k_i$ are real constants, is a four-parameter family of solution of the coupled KdV system related with $(u,v)$ by the B\"{a}cklund transformation \eqref{eq:BT}.

Indeed, using Theorem \ref{th:BT}, the system \eqref{eq:BT}, reduces to
\begin{equation}\label{eq:BTtrivial}
    \begin{array}{l}
         u'_x=u'\sigma\sqrt{u'v'-\lambda_1^2}+\lambda_2 u', \\
         v'_x=v'\sigma\sqrt{u'v'-\lambda_1^2}-\lambda_2 v', \\
         u'_t=u'(\lambda_1^2-3\lambda_2^2)\sigma\sqrt{u'v'-\lambda_1^2}-u'(\lambda_2^3-3\lambda_2\lambda_1^2),\\
         v'_t=v'(\lambda_1^2-3\lambda_2^2)\s\sqrt{u'v'-\lambda_1^2}+v'(\lambda_2^3-3\lambda_2\lambda_1^2).
    \end{array}
\end{equation}
We combine the first two equations and the last two equations of the system to obtain the following equations
\begin{equation*}
        \left(\frac{v'}{u'}\right)_x=-2 \left(\frac{v'}{u'}\right)\lambda_2,  \qquad
        \left(\frac{v'}{u'}\right)_t=2 \left(\frac{v'}{u'}\right)(\lambda_2^3-3\lambda_2\lambda_1^2),
\end{equation*}
By solving these two equations for $v'/u'$, we conclude
\begin{equation}\label{eq:u1v1r}
    v'=u'e^{2\left(-\lambda_2 x+(\lambda_2^3-3\lambda_2\lambda_1^2)t+k_2\right)},
\end{equation}
where $k_2$ is a constant. From the last equation, the system \ref{eq:BTtrivial} reduces to,
\begin{equation}\label{eq:sysu1}
    \begin{array}{l}
        u'_x=u'\sigma\sqrt{\left(u'e^{\left(-\lambda_2 x+(\lambda_2^3-3\lambda_2\lambda_1^2)t+k_2\right)}\right)^2-\lambda_1^2} +u'\lambda_2, \\
        u'_t=u'(\lambda_1^2-3\lambda_2^2)\sigma\sqrt{\left(u'e^{\left(-\lambda_2 x+(\lambda_2^3-3\lambda_2\lambda_1^2)t+k_2\right)}\right)^2-\lambda_1^2}-u'(\lambda_2^3-3\lambda_2\lambda_1^2),
    \end{array}
\end{equation}
We introduce a new variable $U(x,t)$, defined by \begin{equation}\label{eq:Udef}
U(x,t)^2=\left(u'(x,t)e^{\left(-\lambda_2 x+(\lambda_2^3-3\lambda_2\lambda_1^2)t+k_2\right)}\right)^2-\lambda_1^2.
\end{equation}
Hence, system \eqref{eq:sysu1} reads
\begin{equation*}
    \begin{array}{l}
         U_x=\sigma(U^2+\lambda_1^2),\\
         U_t=\sigma(\lambda_1^2-3\lambda_2^2)(U^2+\lambda_1^2), 
    \end{array}
\end{equation*}
which implies that
\begin{equation}\label{eq:U}
    U=\sigma\lambda_1\tan\left(\lambda_1 x+(\lambda_1^3-3\lambda_1\lambda_2^2)t+k_1\right),
\end{equation}
where $k_1$ is a constant. Therefore, replacing \eqref{eq:U} into \eqref{eq:Udef} and considering \eqref{eq:u1v1r}, we have Equation \eqref{eq:BTuv0}.
\end{Example}
To compute the B\"{a}cklund transformation for a solution of the coupled KdV system, one must solve the first-order system \eqref{eq:BT}. However, the proof of Theorem \ref{th:BT} provides an alternative approach to obtaining such a B\"{a}cklund transformation that might be simpler in certain situations, by solving the first-order system \eqref{eq:psiphi} for the functions $\phi$ and $\psi$, as indicated in the next theorem.
\begin{Theorem}\label{th:BTalt}
    Let $(u,v)$ be a real solution of the coupled KdV system. Then for any real constants $\lambda_1$ and $\lambda_2$, the system \eqref{eq:psiphi} is completely integrable for $\phi,\psi$. Furthermore, $(u',v')$ defined by
\begin{equation}\label{eq:BTalt}
    u'=u+\frac{\lambda_1}{\psi},\qquad v'=v+\lambda_1\frac{\phi^2+\psi^2}{\psi},
\end{equation}
where $\phi,\psi$ satisfy \eqref{eq:psiphi}, is also a solution of the coupled KdV system.
\end{Theorem}
\begin{Example}\label{ex:BT2}
We consider $\lambda_1$ and $\lambda_2$ as constants, and $u \equiv 0$ and $v \equiv 1$ as a solution of the coupled KdV system. Then, 
    \begin{equation}\label{eq:BTEx2}
             u'=-\dfrac{\lambda_1(\lambda_1^2+\lambda_2^2)}{\lambda_1+e^{\xi_2}\sin\xi_1},\qquad 
v'=-\dfrac{e^{\xi_2}}{\lambda_1^2+\lambda_2^2}\dfrac{\lambda_1e^{\xi_2}+(\lambda_1^2-\lambda_2^2)\sin\xi_1+2\lambda_1\lambda_2\cos\xi_1}{\lambda_1+e^{\xi_2}\sin\xi_1},
    \end{equation}\
where, $\xi_1=\lambda_1x+(\lambda_1^3-3\lambda_1\lambda_2^2)t+k_1$, $\xi_2=-\lambda_2 x+(\lambda_2^3-3\lambda_2\lambda_1^2)t+k_2$ with $k_1$ and $k_2$ being constants, is a four-parameter family of solutions to the coupled KdV system, related by a B\"{a}cklund transformation to the solution $(u,v)$.
Indeed, in this case, the equation \eqref{eq:psiphi} reduces to
\begin{equation*}
\begin{split}
\phi_x=&\lambda_1\psi-\lambda_2\phi+1,\\
\psi_x=&-\lambda_1\phi-\lambda_2\psi,\\
\phi_t=& (\lambda_2^3-3\lambda_2\lambda_1^2)\phi+(\lambda_1^3-3\lambda_1\lambda_2^2)\psi +\lambda_1^2-\lambda_2^2,\\
\psi_t=&-(\lambda_1^3-3\lambda_1\lambda_2^2)\phi+(\lambda_2^3-3\lambda_2\lambda_1^2)\psi-2\lambda_1\lambda_2,
\end{split}
\end{equation*}
from which we obtain,
\begin{equation*}
    \phi=\dfrac{e^{\xi_2}\cos\xi_1+\lambda_2}{\lambda_1^2+\lambda_2^2},\qquad \psi=-\dfrac{e^{\xi_2}\sin(\xi_1)+\lambda_1}{\lambda_1+\lambda_2},
\end{equation*}
where, $\xi_1=\lambda_1x+(\lambda_1^3-3\lambda_1\lambda_2^2)t+k_1$, $\xi_2=-\lambda_2 x+(\lambda_2^3-3\lambda_2\lambda_1^2)t+k_2$ with $k_1$ and $k_2$ being constants of integration. Therefore, from Theorem \ref{th:BTalt} we obtain Equation \eqref{eq:BTEx2}.
\end{Example}
\section*{Acknowledgment}

The author acknowledges the \textit{Institute of Mathematics of the University of Granada} (IMAG) Grants Program, \textit{Support for Visits of Young Talented Researchers} (2023 call), for its support. Special thanks are extended to Miguel S\'anchez and Joaqu\'in P\'erez, who both served as Directors of IMAG during the stay, and to Antonio Mart\'inez L\'opez for hosting the visit.


\begin{thebibliography}{99}
\bibitem {akns} M. J. Ablowitz, D. J. Kaup, A. Newell, H. Segur, The inverse scattering transform Fourier analysis for nonlinear problems, Stud. Appl. Math. 53 (1974) 249-315.
\bibitem {adler2000} V. E. Adler, A. B. Shabat, R. I. Yamilov, Symmetry approach to the integrability problem, Theoretical and Mathematical Physics, Vol. 125, No. 3 (2000) 1603-1661.
\bibitem{Backlund1985} A. V. B\"acklund, Einiges \"uber Curve und Fl\"achentransformationen, Lund Universi\"et Arsskrift 10, 1985.
\bibitem{Backlund1905} A. V. B\"acklund, Concerning surfaces with constant negative curvature, Translated by E. M. Coddington, New Era Printing Co., Lancaster pa, 1905.
\bibitem{BarbosaFerreiraTenenblat} J. L. Barbosa, W. Ferreira and K. Tenenblat, Submanifolds of constant sectional curvature in pseudo-Riemannian manifolds, Ann. Global Analysis and Geometry 14 (1996)  381-401.
\bibitem {beals1989} R. Beals, M. Rabelo, K. Tenenblat, B{\"a}cklund transformations and inverse scattering for some pseudo-spherical surface equations, Stud. Appl. Math. 81 (1998) 125-151.
\bibitem{BealsTenenblat1991} R. Beals, K. Tenenblat, An intrinsic generalization for the wave and sine-Gordon equation, Differential Geometry, Lawson,B. et all Editors. Pitman Monographs and Surveys in Pure and Applied Math. 52 (1991) 25-46.
\bibitem{bianchi1} L. Bianchi, Sulla transformazione di B\"acklund per le superficie pseudosferiche, Rend. Acc. Naz. Lincei 5 (1892) 3-13.
\bibitem{Campos} P. T. Campos, B\"acklund transformations for the generalized Laplace and elliptic sinh-Gordon equations, Anais da Acad. Bras. de Ci\^encias (1994) 405-411. 
\bibitem{CamposTenenblat} P. T. Campos, K. Tenenblat, B\"acklund transformations for a class of systems of differential equations, Geometric and Functional Analysis 4 (1994) 270-287.
\bibitem {castrosilva2015} T. Castro Silva, K. Tenenblat, Third order differential equations describing pseudospherical surfaces, J. Differential Equations 259 (2015) 4897-4923.
\bibitem {Catalano2020} D. Catalano Ferraioli, T. Castro Silva, K. Tenenblat, A class of quasilinear second order partial differential equations which describe spherical or pseudospherical surfaces, Journal of Differential Equations, Volume 268, Issue 11 (2020) 7164-7182.
\bibitem {Catalano2016} D. Catalano Ferraioli, L. A, de Oliveira Silva, Second order evolution equations which describe pseudospherical surfaces, J. Differential Equations 260 (2016) 8072-8108.
\bibitem {Catalano2024} D. Catalano Ferraioli, T. Castro Silva, A class of third order quasilinear partial differential equations describing spherical or pseudospherical surfaces,
J. of Differential Equations, Volume 379 (2024) 524-568.
\bibitem {Catalano2014} D. Catalano Ferraioli, K. Tenenblat, Fourth order evolution equations wich describe pseudospherical surfaces, J. Differential Equations 257 (2014) 3165-3199.
\bibitem {Cavalcante1988} J. Cavalcante, K. Tenenblat, Conservation laws for nonlinear evolution equations, J. Math. Phys. 29 (1988) 1044-1049.
\bibitem{ChenZuoCheng2004} Q. Chen, D. Zuo, Y.Cheng, Isometric immersions of pseudo-Riemannian space forms, J. Geom. Phys. 52, 3 (2004) 241-262.
\bibitem {chern} S. S. Chern, K. Tenenblat, Pseudo-Spherical surfaces and evolution equations, Stud. Appl. Math. 74 (1986) 55-83.
\bibitem {DajczerTojeiro} M. Dajczer, R. Tojeiro, Isomeytric immersions and the generalized Laplace and elliptic sinh-Gordon equations, Journal f\"ur die Reine und Angewandte Mathematik,  467 (1995) 109-147.
\bibitem {ding} Q. Ding, K. Tenenblat, On differential systems describing surfaces of constant curvature, J. Differential Equations 184 (2002) 185-214.
\bibitem {gomesneto2010} V. P. Gomes Neto, Fifth-order evolution equations describing pseudospherical surfaces, J. Differential Equations 249 (2010) 2822-2865.
\bibitem{Gu-Hu-Inoguchi2002} C.H. Gu, H.S. Hu, J. Inoguchi, \textit{On time-like surfaces of positive constant gaussian curvature and imaginary principals curvatures}, Elsevier, Journal of Geometry and Physics 41 (2002) 296-311.
\bibitem{Hirota1981} R. Hirota, J. Satsuma, Soliton solutions of a coupled Korteweg-de Vries equation, Physics Letters A, 85.8-9 (1981) 407-408.
\bibitem {ivey} T. A. Ivey, J. M. Landsberg, \textit{Cartan for Beginners: Differential Geometry via Moving Frames and Exterior Differential Systems}, Graduate studies in mathematics, ISSN 1065-7339; v. 61, 2003.
\bibitem {Jorge1987} L. P. Jorge, K. Tenenblat, Linear problems associated to evolution equations of type $u_{tt}=F(u,u_{x},u_{xx},u_{ut})$, Stud. Appl. Math. 77 (1987) 103-117.
\bibitem {kamran1995} N. Kamran, K. Tenenblat, On differential equations describing pseudospherical surfaces, J. Differential Equations 115 (1995) 71-98.
\bibitem{KelmerRodrigues2022} F. Kelmer, L. A. Rodrigues, K. Tenenblat, A unified approach to B\"acklund type theorems for surfaces in 3-dimensional pseudo-euclidean space, Matem\'atica Contempor\^anea, Vol. 49 (2022) 140-170.
\bibitem {kelmer2022} F. Kelmer, K. Tenenblat,
On a class of systems of hyperbolic equations describing pseudo-spherical or spherical surfaces, J. of Differential Equations, Volume 339 (2022) 372-394.
\bibitem {KelmerTenenblat2024} F. Kelmer, K. Tenenblat, Superposition Formulae for the Geometric B\"acklund Transformations of the Hyperbolic and Elliptic Sine-Gordon and Sinh-Gordon Equations, SIGMA 20, 015 (2024) 27 pages.
\bibitem{KodamaHasegawa1987} Y. Kodama, A. Hasegawa, Nonlinear pulse propagation in a monomode dielectric guide. IEEE Journal of Quantum Electronics. 23.5 (1987) 510-524.
\bibitem{McNertney1980} L. V. McNertney, \textit{One-parameter families of surfaces with constant curvature in Lorentz 3-space}, Ph. D. thesis, Brown University, 1980.
\bibitem{Nucci1987} M. C. Nucci, Pseudopotentials, Lax equations and B\"{a}cklund transformations for non-linear evolution equations, J. Phys. A. Math. Gen. 21 (1988) 73-79.
\bibitem {Palmer1990} B. Palmer, \textit{B\"acklund transformations for surfaces in Minkowski space}, J. Math. Phys. 31 (1990) 2872.
\bibitem {Rabelo1989} M. Rabelo, On equations which describe pseudo-spherical equations, Stud. Appl. Math. 81 (1989) 221-248.
\bibitem {Rabelo1990} M. Rabelo, K. Tenenblat, On equations of type $u_{xt}=F(u,u_x)$ which describe pseudo-spherical surfaces, J. Math. Phys. 31 (1990) 1400-1407.
\bibitem {Rabelo1992} M. Rabelo, K. Tenenblat, A classification of pseudo-spherical surface equations of type $u_t=u_{xxx}+G(u,u_x,u_{xx})$, J. Math. Phys. 33 (1992), 537-549.
\bibitem {reyes1998} E. G. Reyes, {Pseudo-spherical surfaces and integrability of evolution equations}, J. Differential Equations 147 (1) (1998) 195-230, Erratum: J. Differential Equations 153 (1) (1999) 223-224.
\bibitem{Reyes-backlund} E. G. Reyes, On generalized B{\"a}cklund transformations
for equations describing pseudo-spherical surfaces, J. Geom. Phys.
{45} (3-4) (2003) 368-392. 
\bibitem{Rey7} E. G. Reyes, Nonlocal symmetries and the Kaup\textendash{}Kupershmidt
equation, J. Math. Phys. 46 (7), (2005) 073507. 
\bibitem{Rey8}E. G. Reyes, Pseudo-potentials, nonlocal symmetries, and integrability of some shallow water equations, Select.
Math. (New Series) (2006) 241\textendash{}270.
\bibitem {sasaki} R. Sasaki, Soliton equations and pseudospherical surfaces, Nucl. Phys. B, 154 (1979) 343-357.
\bibitem {keti} K. Tenenblat, Transformations of Manifolds and Applications to Differential Equations Pitman Monographs and Surveys in Pure and Applied Mathematics 93, ISBN 0 582 316197, 1998.
\bibitem {Tenenblat1985} K. Tenenblat, B\"acklund theorem for submanifolds of space forms and a generalized wave equation, Bol. Soc. Bras. Mat. 16 (1985) 69-94.
\bibitem {TenenblatTerng} K. Tenenblat, C. L. Terng, B\"acklund's theorem for $n$-dimensional submanifolds of $R^{2n-1}$, Ann. Math. 111 (1980) 477-490.
\bibitem {Terng} C. L. Terng, A higher dimensional generalization of the sine-Gordon equation and its soliton theory, Ann. Math. 111 (1980) 491-510. 

\end{thebibliography}
\end{document}